\documentclass{amsart}

\usepackage{amsmath}
\usepackage{amssymb}
\usepackage{amsthm}
\usepackage{graphicx}
\usepackage{enumitem}
\usepackage{hyperref}

\theoremstyle{plain}
\newtheorem{thm}{Theorem}[section]
\newtheorem{prop}[thm]{Proposition}
\newtheorem{lem}[thm]{Lemma}

\theoremstyle{definition}
\newtheorem{defn}[thm]{Definition}
\newtheorem{claim}{Claim}[thm]
\newtheorem{question}{Question}

\newcommand{\id}{\mathrm{id}}

\begin{document}

\title[Combinatorial $n$-od covers]{Combinatorial $n$-od covers of graphs}
\author{Logan C. Hoehn \and Hugo Adrian Maldonado-Garcia}
\date{\today}

\address{Nipissing University, Department of Computer Science \& Mathematics, 100 College Drive, Box 5002, North Bay, Ontario, Canada, P1B 8L7}
\email{loganh@nipissingu.ca}

\address{Instituto de Matem\'{a}ticas, Universidad Nacional Aut\'{o}noma de M\'{e}xico, \'{A}rea de la Investigaci\'{o}n Cient\'{i}fica, Circuito exterior, Ciudad Universitaria, 04510, M\'{e}xico, CDMX}
\email{hugoadrmg@gmail.com}

\thanks{This work was supported by NSERC grant RGPIN-2019-05998}


\begin{abstract}
We introduce the notion of a combinatorial $n$-od cover, for $n \geq 3$, which is a tool that may be used to show that certain continua embedded in the plane are not simple $n$-od-like.  Using this tool, we generalize a classic example of Ingram, and give a construction, for each $n \geq 3$, of an indecomposable plane continuum which is simple $(n+1)$-od-like but not simple $n$-od-like, and such that each non-degenerate proper subcontinuum is an arc.  These examples may be compared with related constructions of Kennaugh \cite{kennaugh2009}.
\end{abstract}

\maketitle

\section{Introduction}
\label{sec:intro}

This paper concerns the complexity of some tree-like continua in the plane $\mathbb{R}^2$, as measured by the simplest tree $T$ for which they are $T$-like, if such a tree exists.  Here, a \emph{continuum} is a non-empty compact, connected, metric space.  Recall that a continuum $X$ is \emph{tree-like} (respectively \emph{$T$-like}, where $T$ is a tree) if for any $\varepsilon > 0$, there exists an open cover of $X$ with mesh less than $\varepsilon$ whose nerve is homeomorphic to a tree (respectively, to a subcontinuum of $T$).  A \emph{simple $n$-od} is a tree that has exactly one branch point, which has degree $n$.  See Section~\ref{sec:prelim} below for more definitions of standard terms used throughout this paper.

The study of tree-like plane continua is relevant to some of the oldest and most significant open questions in continuum theory, see e.g.\ \cite[Epilogue]{bellamy1980}, \cite[Question~1]{hoehn-oversteegen2016}, and \cite[Introduction]{hoehn-oversteegen2020}.  In the spirit of understanding ``minimally complex'' tree-like plane continua, W.\ Lewis posed two questions in \cite{lewis2007} related to this subject.  Question 26 of \cite{lewis2007} asks whether, for each $n \geq 2$, there exists a simple $(n+1)$-od-like plane continuum which is not simple $n$-od-like, and such that each proper subcontinuum is an arc, where an \emph{arc} is a space which is homeomorphic to the closed interval $[0,1]$.  We answer this question in the affirmative in this paper.  The more general Question 27 of \cite{lewis2007} asks whether, for any graph $G$, there exists an atriodic $G$-like continuum that is not $H$-like for any graph $H$ which is a monotone image of $G$ but not vice versa.  Recall that a continuum $X$ is \emph{atriodic} if it contains no triod; that is, if $A \subset B \subseteq X$ are subcontinua, then $B \smallsetminus A$ has at most two components.

A continuum is \emph{indecomposable} if it is not the union of two of its proper subcontinua, and an \emph{arc continuum} is a continuum each of whose proper non-degenerate subcontinua is homeomorphic to an arc.  By a \emph{proper subcontinuum}, we mean a continuum $Y \subset X$ such that $Y \neq X$, and a continuum is \emph{non-degenerate} if it contains more than one point.  Note that each indecomposable arc continuum is, in particular, atriodic.  There are many interesting indecomposable tree-like arc continua in the literature; we list a few examples here:
\begin{enumerate}
\item W.\ T.\ Ingram \cite{ingram1972} constructed an indecomposable arc continuum in the plane which is simple triod-like and not arc-like (in fact, has positive span).  Subsequent variations of this example include an uncountable family of hereditarily indecomposable simple triod-like continua in the plane, such that every proper subcontinuum is a pseudoarc, and which are not arc-like (in fact, have positive span) \cite{ingram1979}.
\item P.\ Minc \cite{minc1993}, constructed an indecomposable arc continuum which is simple $4$-od-like but not simple triod-like.
\item C.\ T.\ Kennaugh \cite{kennaugh2009}, in his Ph.D.\ dissertation (not published), gave a construction, for every $n \geq 3$, of an indecomposable arc continuum which is simple $(n+1)$-od-like but not simple $n$-od-like.  These continua are not known to embed in the plane.
\item L.\ C.\ Hoehn \cite{hoehn2011} gave an example of an indecomposable simple triod-like arc continuum in the plane which is not arc-like, and which has span zero.
\end{enumerate}

In this paper, we introduce the notion of a combinatorial $n$-od cover, for $n \geq 2$, which is a tool that may be used to show that certain continua embedded in the plane are not simple $n$-od-like.  This is a generalization of the notion of a ``chain quasi-order'', defined in \cite{hoehn2011}.  Our work suggests that similar notions can be developed to show that certain continua are not $T$-like, for a given tree $T$ (see the discussion in Section~\ref{sec:questions} below).

As an application, we generalize the classic example of Ingram \cite{ingram1972}, and give a construction of a family of continua in the plane with the same properties as in \cite{kennaugh2009}.  In a forthcoming paper, we will use the same framework to construct such examples which also have span zero.

Many of the notions developed in this paper are adapted from corresponding ideas from \cite{hoehn2011}, and we will highlight the correspondences throughout.

\section{Preliminaries}
\label{sec:prelim}

Given two points $x_1,x_2 \in \mathbb{R}^2$, the distance between them is denoted by $\|x_1 - x_2\|$.

\subsection{Graphs}
\label{sec:graphs}

In this paper, a (combinatorial) graph $G$ consists of a set of vertices $\mathsf{V}(G)$ and a set of edges $\mathsf{E}(G)$.  An edge is an unordered pair $\{u,v\}$ of vertices.  Two vertices $u,v \in \mathsf{V}(G)$ are \emph{adjacent} if there is an edge between them, i.e.\ if $\{u,v\} \in \mathsf{E}(G)$.  The \emph{degree} of a vertex is the number of vertices which are adjacent to it.  Vertices $v_1,\ldots,v_m \in \mathsf{V}(G)$ are \emph{consecutive} in $G$ if $v_i$ and $v_{i+1}$ are adjacent for each $i \in \{1,\ldots,m-1\}$.  $G$ is \emph{connected} if for any distinct vertices $u,w \in \mathsf{V}(G)$, there are consecutive vertices $v_1,\ldots,v_m$ in $G$ such that $v_1 = u$ and $v_m = w$.  All graphs considered in this paper will be connected.

If $S \subseteq \mathsf{V}(G)$, the \emph{subgraph generated by $S$} is the graph $G'$ whose vertex set is $\mathsf{V}(G') = S$, and two vertices in $G'$ are adjacent if and only if they are adjacent in $G$.  Given a vertex $v \in \mathsf{V}(G)$, we define the subgraph $G - v$ to be the subgraph generated by the set $\mathsf{V}(G) \smallsetminus \{v\}$.  If $G$ is a connected graph with at least two vertices, then there exist at least two vertices $v \in \mathsf{V}(G)$ such that $G - v$ is also connected.  We define edge removal in a somewhat non-standard way: given an edge $e = \{u,v\} \in \mathsf{E}(G)$, the subgraph $G - e$ is obtained by removing the edge $e$ from $G$, and also removing the vertices $u$ and $v$ if $e$ was the only edge connected to them.  Precisely, $\mathsf{E}(G - e) = \mathsf{E}(G) \smallsetminus \{e\}$, $\mathsf{V}(G) \smallsetminus \{u,v\} \subseteq \mathsf{V}(G - e) \subseteq \mathsf{V}(G)$, and $u \in \mathsf{V}(G - e)$ (respectively, $v \in \mathsf{V}(G - e)$) if and only if it is adjacent to another vertex other than $v$ (respectively, other than $u$) in $G$.  If $G$ is connected and has at least one edge, then there exists at least one edge $e \in \mathsf{E}(G)$ such that $G - e$ is connected.


Let $n \geq 3$.  An \emph{$n$-od graph} is a connected graph $G$ with exactly one vertex $o_G$ of degree $n$, called the \emph{branch vertex} of $G$, and all other vertices have degree $2$ or $1$.  The vertices in the connected components of $G - o_G$, where $o_G$ is the branch vertex of $G$, are called the \emph{legs} of $G$.  We assume that the legs come with a fixed enumeration, either $1$ through $n$ or $0$ through $n-1$.  For each leg number $\ell$, we denote the vertices on leg $\ell$ by $G(\ell,1), G(\ell,2), \ldots$, listed in order from nearest to the branch outwards.  We allow the possibility that the legs of $G$ have infinitely many vertices (as with the infinite $n$-od $C$ below, in Section~\ref{sec:comb covers}).  We will denote the branch vertex of $G$ by $G(i,0)$ for any leg number $i$, or by $G(\cdot,0)$.  A \emph{$2$-od graph} (or \emph{arc graph}) is a connected graph each of whose vertices has degree $2$ or $1$.

A \emph{topological graph} is a continuum which is the union of finitely many arcs, each two of which have finite (possibly empty) intersection.  Any finite combinatorial graph can be viewed as a topological graph by starting with the discrete vertex set $\mathsf{V}(G)$, then joining each pair of adjacent vertices $u,v \in \mathsf{V}(G)$ by an arc, which we denote by $uv$.  A \emph{simple $n$-od} is a graph which is the union of $n$ arcs, which all have one endpoint in common, and are otherwise pairwise disjoint.  Note that a finite $n$-od (combinatorial) graph, when viewed as a topological graph, is a simple $n$-od.

\subsection{Covers}
\label{sec:covers}

Let $X$ be a continuum.  Given an open cover $\mathcal{U}$ of $X$, the \emph{mesh} of $\mathcal{U}$ is the supremum of the diameters of elements of $\mathcal{U}$.  A continuum $X$ is $1$-dimensional if and only if for any $\varepsilon > 0$, there is an open cover $\mathcal{U}$ of $X$ such that any three distinct elements of $\mathcal{U}$ have empty intersection.  Given such a cover $\mathcal{U}$, the \emph{nerve} of $\mathcal{U}$ is the graph $\mathsf{N}(\mathcal{U})$ whose vertex set is $\mathcal{U}$, and two elements $U_1,U_2 \in \mathcal{U}$ are adjacent if and only if $U_1 \cap U_2 \neq \emptyset$.

Let $n \geq 2$.  By an \emph{$n$-od cover} of $X$, we mean an open cover $\mathcal{U}$ of $X$ whose nerve is an $n$-od graph.  A continuum $X$ is \emph{simple $n$-od-like} if for each $\varepsilon > 0$, there exists an $n$-od cover $\mathcal{U}$ of $X$ with mesh less than $\varepsilon$.  If $\mathcal{U}$ is an $n$-od cover of $X$, then we adopt a notation similar to what we use for combinatorial $n$-od graphs, to refer to the elements of $\mathcal{U}$.  Specifically, we assume the legs of $\mathcal{U}$ come with a fixed enumeration $1,\ldots,n$.  For each leg $\ell \in \{1,\ldots,n\}$, we denote the elements on leg $\ell$ by $U(\ell,1), U(\ell,2), \ldots$, listed in order from nearest to the branch outwards.  We will denote the element of $\mathcal{U}$ which is the branch vertex of $\mathsf{N}(\mathcal{U})$ by $U(\cdot,0)$, or by $U(i,0)$ for any leg $i \in \{1,\ldots,n\}$, whenever it is convenient to do so.

Let $\mathcal{U}$ be an $n$-od cover (or more generally, an open cover whose nerve is a tree).  Given two distinct elements $V,W \in \mathcal{U}$, the \emph{chain} in $\mathcal{U}$ from $W_1$ to $W_2$ is the unique sequence of distinct elements $U_1,\ldots,U_m \in \mathcal{U}$ such that $U_1 = V$, $U_m = W$, and $U_i \cap U_{i+1} \neq \emptyset$ for each $i \in \{1,\ldots,m-1\}$.  If $A \subset X$ is an arc and $V \cap A \neq \emptyset \neq W \cap A$, then each element in the chain in $\mathcal{U}$ from $V$ to $W$ must intersect $A$.


\section{Plane embeddings of graphs}
\label{sec:embeddings}

The continua constructed in this paper are obtained as limits of sequences of graphs embedded in the plane.  We develop notions in this section to facilitate the precise placement of graphs in the plane.

For the remainder of this paper, let $n \geq 2$ be a fixed integer.  Let $o$ denote the origin of the plane.  We will work with the $(n+1)$-od $T_0$ in $\mathbb{R}^2$ whose branch point is $o$, and whose legs are the straight line segments in $\mathbb{R}^2$ joining $o$ to the points $\left( 2 \cos \left( \frac{i\pi}{n} \right), 2 \sin \left( \frac{i\pi}{n} \right) \right)$, for $i = 0,\ldots,n$.

Define the function $b \colon \{0,\ldots,n\} \times [0,1] \to \mathbb{R}^2$ by
\[ b(i,t) = \left( (1+t) \cdot \cos(\tfrac{i\pi}{n}), (1+t) \cdot \sin(\tfrac{i\pi}{n}) \right) .\]
Define the set
\[ \Gamma = \{o\} \cup \{b(i,t): i \in \{0,\ldots,n\}, t \in [0,1]\} .\]
See Figure~\ref{fig:T_0} for an illustration of $T_0$ and the points of the form $b(i,t)$, for the case $n = 3$.

\begin{figure}
\begin{center}
\includegraphics{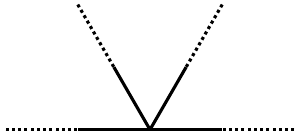}
\end{center}

\caption{The $(n+1)$-od $T_0$ in the case $n = 3$.  The set of points of the form $b(i,t)$ is represented by the dotted lines.}
\label{fig:T_0}
\end{figure}

Given an $(n+1)$-od $T \subset \mathbb{R}^2$ with branch point $v$, we say the legs of $T$ are in \emph{proper circular order} if for any neighborhood $U$ of $v$ in $\mathbb{R}^2$, there exists a path in $U \smallsetminus \{v\}$ which meets each leg in exactly one point, and the intersection with leg $i$ precedes the intersection with leg $j$ along the path whenever $i < j$.  For any two $(n+1)$-ods $T_1,T_2 \subset \mathbb{R}^2$ whose legs are in proper circular order, there is a homeomorphism $\Psi \colon \mathbb{R}^2 \to \mathbb{R}^2$ such that $\Psi(T_1) = T_2$ (and in fact $\Psi$ takes leg $i$ of $T_1$ to leg $i$ of $T_2$, for each $i = 0,\ldots,n$).

The next three Definitions below are adaptations of the notions of a ``$\Gamma$-marking'', a ``compliant graph-word'', and a ``$\langle T,\varepsilon \rangle$-sketch'', from \cite[Section 2.1]{hoehn2011}, to the present setting.

\begin{defn}
\label{defn:Gamma space}
A \emph{$\Gamma$-space} is a set $T \subset \mathbb{R}^2$, together with a function $\mathfrak{m} \colon \Gamma \to T$, called the \emph{$\Gamma$-marking of $T$}, such that:
\begin{enumerate}[label=(\arabic{*})]
\item $T$ is an $(n+1)$-od whose legs are in proper circular order;
\item $\mathfrak{m}(o)$ is the branch of $T$; and
\item for each $i \in \{0,\ldots,n\}$, $\mathfrak{m}(b(i,1))$ is the endpoint of leg $i$ of $T$ and the map $t \mapsto \mathfrak{m}(b(i,t))$, $t \in [0,1]$, isometrically parameterizes a straight arc in leg $i$ of $T$ which does not contain the branch of $T$.
\end{enumerate}
\end{defn}

Observe that $T_0$ is a $\Gamma$-space, with $\Gamma$-marking $\id_\Gamma$.

Throughout the remainder of this paper, we will abbreviate the phrase ``$x = b(i,t)$ for some $t \in [0,1]$'' by writing simply ``$x = b(i,\cdot)$''.

\begin{defn}
\label{defn:placement function}
Let $G$ be a graph.  A \emph{placement function} is a function $\omega \colon \mathsf{V}(G) \to \Gamma$ such that whenever $u,v \in \mathsf{V}(G)$ are adjacent vertices, we have either $\omega(u) = o$ and $\omega(v) = b(i,\cdot)$, or $\omega(v) = o$ and $\omega(u) = b(i,\cdot)$, for some $i \in \{0,\ldots,n\}$.
\end{defn}

\begin{defn}
\label{defn:embedding}
Let $T \subset \mathbb{R}^2$ be a $\Gamma$-space with $\Gamma$-marking $\mathfrak{m}_T \colon \Gamma \to T$, let $G$ be a graph, and let $\omega \colon \mathsf{V}(G) \to \Gamma$ be a placement function.  Given $\varepsilon > 0$, a \emph{$\langle \omega,T,\varepsilon \rangle$-embedding} of $G$ is an embedding $\Omega \colon G \to \mathbb{R}^2$ together with a projection $\pi \colon \Omega(G) \to T$ such that:
\begin{enumerate}[label=(\arabic{*})]
\item for each vertex $u \in \mathsf{V}(G)$, $\pi(\Omega(u)) = \mathfrak{m}_T(\omega(u))$;
\item for each edge $\{u,v\}$ in $G$, the restriction of $\pi$ to the arc in $\Omega(G)$ joining $\Omega(u)$ and $\Omega(v)$ corresponding to this edge is a homeomorphism to the arc in $T$ joining the points $\mathfrak{m}_T(\omega(u))$ and $\mathfrak{m}_T(\omega(v))$; and
\item for any $x \in \Omega(G)$, the distance from $x$ to $\pi(x)$ is less than $\varepsilon$.
\end{enumerate}
\end{defn}

The following result is included for practical purposes.  It shows that in order to construct $\langle \omega,T,\varepsilon \rangle$-embeddings of a graph $G$, where $\omega \colon \mathsf{V}(G) \to \Gamma$ is a placement function and $T$ is a $\Gamma$-space, it suffices to exhibit $\langle \omega,T_0,\varepsilon \rangle$-embeddings of $G$, where $T_0$ is the $\Gamma$-space described above.

\begin{prop}
\label{prop:embed G}
Let $T \subset \mathbb{R}^2$ be a $\Gamma$-space with $\Gamma$-marking $\mathfrak{m}_T$, let $G$ be an $(n+1)$-od graph and let $\omega \colon \mathsf{V}(G) \to \Gamma$ be a placement function.  Suppose that for any $\varepsilon > 0$, there exists a $\langle \omega,T_0,\varepsilon \rangle$-embedding $\Omega_0$ of $G$ with projection $\pi_0 \colon \Omega_0(G) \to T_0$ such that:
\begin{enumerate}[label=(\roman{*})]
\item $\Omega_0(G)$ is a $\Gamma$-space with $\Gamma$-marking $\mathfrak{m}_0$; and
\item for each $i \in \{0,\ldots,n\}$, there exists $i' \in \{0,\ldots,n\}$ such that for all $t \in [0,1]$, $\pi_0 \circ \mathfrak{m}_0(b(i,t)) = b(i',t)$.
\end{enumerate}
Then for any $\varepsilon > 0$, there exists a $\langle \omega,T,\varepsilon \rangle$-embedding $\Omega$ of $G$ with projection $\pi \colon \Omega(G) \to T$ such that:
\begin{enumerate}[label=(\arabic{*})]
\item $\Omega(G)$ is a $\Gamma$-space with $\Gamma$-marking $\mathfrak{m}$; and
\item for each $i \in \{0,\ldots,n\}$, there exists $i' \in \{0,\ldots,n\}$ such that for all $t \in [0,1]$, $\pi \circ \mathfrak{m}(b(i,t)) = \mathfrak{m}_T(b(i',t))$.
\end{enumerate}
\end{prop}

The basic idea of the proof of Proposition~\ref{prop:embed G} is to choose a homeomorphism of the plane taking $T_0$ to $T$, then compose it with an embedding of $G$ very close to $T_0$ to obtain the desired embedding of $G$ close to $T$.  The only technical issue in this plan is that we need to ensure that the ends of the legs of $G$ are embedded as straight-line segments, and that condition (2) is satisfied.  For this, we first formulate and prove the following technical result, whose statement and proof may be of independent worth, although it is not at all surprising.  It states that for any free straight-line arc in a plane graph, we may clear out the rest of the graph from a rectangular neighborhood of the arc with an arbitrarily small motion of the plane near the endpoints of the arc.

Given a graph $H \subset \mathbb{R}^2$ and a straight-line segment $A \subseteq H$, a \emph{rectangular sleeve} for $A$ relative to $H$ is a closed rectangle $R \subset \mathbb{R}^2$ (not necessarily axis-aligned) such that $R \cap H = A$ and $A$ joins the midpoints of two opposite sides of $R$ (i.e.\ $A$ is a ``mid-line'' of $R$).

\begin{lem}
\label{lem:sleeve}
Let $H \subset \mathbb{R}^2$ be a graph and suppose $A \subseteq H$ is a straight-line segment which contains no branch points of $H$, except possibly the endpoints of $A$.  For any $\varepsilon > 0$, there exists a homeomorphism $\Theta \colon \mathbb{R}^2 \to \mathbb{R}^2$ of the plane such that:
\begin{enumerate}[label=(\arabic{*})]
\item for all $x \in \mathbb{R}^2$, the distance from $x$ to $\Theta(x)$ is less than $\varepsilon$;
\item $\Theta(x) = x$ for all $x \in A$ and for all $x \in \mathbb{R}^2$ which are at distance at least $\varepsilon$ from the endpoints of $A$; and
\item there exists a rectangular sleeve for $A = \Theta(A)$ relative to $\Theta(H)$.
\end{enumerate}
\end{lem}

\begin{proof}
We may assume without loss of generality that $A$ is the horizontal arc $[0,1] \times \{0\}$ in $\mathbb{R}^2$.  Let $\varepsilon > 0$.  Choose $0 < \varepsilon' < \min\{\frac{\varepsilon}{4}, \frac{1}{4}\}$, and choose $0 < \delta \leq \varepsilon'$ small enough so that $\left( [\varepsilon', 1 - \varepsilon'] \times [-\delta,\delta] \right) \cap H \subset A$.  We will define a plane homeomorphism $\Theta$ which is non-trivial only in a small neighborhood of the endpoint $(0,0)$ of $A$, for which $\left( [0, 1 - \varepsilon'] \times [-\delta,\delta] \right) \cap \Theta(H) \subset A$.  Doing the same also at the endpoint $(1,0)$ of $A$ yields the desired plane homeomorphism.

Define $g \colon [-\delta,\delta] \to \mathbb{R}$ by
\[ g(y) = \begin{cases}
\sup \{x \in [0,\varepsilon'): (x,y) \in H\} &\textrm{if } y \neq 0 \textrm{ and } [0,\varepsilon'] \times \{y\} \cap H \neq \emptyset \\
0 &\textrm{otherwise.}
\end{cases} \]
By the choice of $\delta$, we have $g(y) < \varepsilon'$ for all $y \in [-\delta,\delta]$.  It is easy to see that $g$ is upper semi-continuous.  Define the lower semi-continuous function $h \colon [-\delta,\delta] \to \mathbb{R}$ by
\[ h(y) = \begin{cases}
\varepsilon' &\textrm{if } y \neq 0 \\
0 &\textrm{if } y = 0 .
\end{cases} \]
Since $g(y) \leq h(y)$ for all $y \in [-\delta,\delta]$, by the Kat\v{e}tov-Tong Insertion Theorem (see e.g.\ \cite{engelking1977}), there exists a continuous function $f \colon [-\delta,\delta] \to \mathbb{R}$ such that $g \leq f \leq h$.  Extend $f$ to $\mathbb{R}$ by setting $f(y) = 0$ for all $y \notin (-2\delta,2\delta)$, and extending $f$ linearly on $(-2\delta,-\delta) \cup (\delta,2\delta)$.

For each $y \in \mathbb{R}$, define $q_y \colon \mathbb{R} \to \mathbb{R}$ by setting $q_y(x) = x - f(y)$ if $x \in [0,\varepsilon']$, $q_y(x) = x$ if $x \notin (-2\varepsilon',2\varepsilon')$, and extending $q_y$ linearly in between these ranges.  Now define $\Theta \colon \mathbb{R}^2 \to \mathbb{R}^2$ by $\Theta(x,y) = (q_y(x),y)$.

It is straightforward to see that $\Theta$ is a homeomorphism of the plane which moves no point more than $\varepsilon$ (in fact $\varepsilon'$), and $\Theta$ is the identity outside of the rectangle $(-2\delta,2\delta) \times (-2\varepsilon',2\varepsilon')$, which is contained in the ball centered at $(0,0)$ of radius $\varepsilon$.  Also, by construction, $\left( [0, 1 - \varepsilon'] \times [-\delta,\delta] \right) \cap \Theta(H) \subset A$.  By performing the analogous construction at the endpoint $(1,0)$ of $A$, and composing the two plane homeomorphisms, we obtain the desired homeomorphism $\Theta \colon \mathbb{R}^2 \to \mathbb{R}^2$ for which $\left( [0,1] \times [-\delta,\delta] \right) \cap \Theta(H) = A$.
\end{proof}

\begin{proof}[Proof of Proposition~\ref{prop:embed G}]
Let $\varepsilon > 0$.  We may assume that the circular order of the legs of $T$ are the same as for $T_0$, so that there is an orientation-preserving homeomorphism of the plane taking $T_0$ to $T$.

For each $i \in \{0,\ldots,n\}$, let $A_i = \{\mathfrak{m}_T(b(i,t)): t \in [0,1]\}$ be the straight segment at the end of leg $i$ of $T$.  According to Lemma~\ref{lem:sleeve}, there exists a homeomorphism $\Theta_1 \colon \mathbb{R}^2 \to \mathbb{R}^2$ of the plane such that:
\begin{enumerate}[label=(\arabic{section}.\arabic{thm}-\roman{*})]
\item for all $x \in \mathbb{R}^2$, the distance from $x$ to $\Theta_1(x)$ is less than $\frac{\varepsilon}{4}$;
\item for each $i \in \{0,\ldots,n\}$ and all $x \in A_i$, $\Theta_1(x) = x$; and
\item for each $i \in \{0,\ldots,n\}$ there exists a rectangular sleeve for $A_i = \Theta_1^{-1}(A_i)$ relative to $\Theta_1^{-1}(T)$.
\end{enumerate}

Now consider a homeomorphism from the union of $T_0$ with some rectangular sleeves $R_i$ at each of the arcs $\{b(i,t): t \in [0,1]\}$, $i \in \{0,\ldots,n\}$, to the union of $\Theta_1^{-1}(T)$ with rectangular sleeves at each of the arcs $A_i = \Theta_1^{-1}(A_i)$, which is an isometry between corresponding rectangular sleeves.  By \cite{adkisson-maclane1940}, this homeomorphism extends to a homeomorphism $\Theta_2 \colon \mathbb{R}^2 \to \mathbb{R}^2$ of the plane.  So, $\Theta_2(T_0) = \Theta_1^{-1}(T)$, and for each $i \in \{0,\ldots,n\}$ and all $t \in [0,1]$, $\Theta_2(b(i,t)) = \mathfrak{m}_T(b(i,t))$.

Let $\varepsilon' > 0$ be small enough so that for any $x \in \mathbb{R}^2$ and $y \in T_0$, if the distance between $x$ and $y$ is less than $\varepsilon'$ then the distance between $\Theta_2(x)$ and $\Theta_2(y)$ is less than $\frac{\varepsilon}{2}$.  By assumption, there exists a $\langle \omega,T_0,\varepsilon' \rangle$-embedding $\Omega_0$ of $G$ with projection $\pi_0 \colon \Omega_0(G) \to T_0$ such that:
\begin{enumerate}[resume, label=(\arabic{section}.\arabic{thm}-\roman{*})]
\item $\Omega_0(G)$ is a $\Gamma$-space with $\Gamma$-marking $\mathfrak{m}_0$; and
\item \label{i'} for each $i \in \{0,\ldots,n\}$, there exists $i' \in \{0,\ldots,n\}$ such that for all $t \in [0,1]$, $\pi_0 \circ \mathfrak{m}_0(b(i,t)) = b(i',t)$.
\end{enumerate}
We may assume further that $\varepsilon'$ is small enough so that $\mathfrak{m}_0(b(i,t)) \in R_{i'}$ for each $i \in \{0,\ldots,n\}$ and all $t \in [0,1]$, where $i'$ is as in \ref{i'}.

Now define $\Omega = \Theta_2 \circ \Omega_0$, $\mathfrak{m} = \Theta_2 \circ \mathfrak{m}_0$, and $\pi = \Theta_1 \circ \Theta_2 \circ \pi_0 \circ \Theta_2^{-1}$.  It should be clear that $\Omega$ is a $\langle \omega,T,\varepsilon \rangle$-embedding of $G$ with projection $\pi$.  Since $\Theta_2$ is an isometry on each of the rectangular sleeves $R_i$, we obtain that $\Omega(G)$ is a $\Gamma$-space with $\Gamma$-marking $\mathfrak{m}$.  For condition (2), we have
\[ \pi \circ \mathfrak{m} = \Theta_1 \circ \Theta_2 \circ \pi_0 \circ \Theta_2^{-1} \circ \Theta_2 \circ \mathfrak{m}_0 = \Theta_1 \circ \Theta_2 \circ \pi_0 \circ \mathfrak{m}_0 .\]
So, given $i \in \{0,\ldots,n\}$ and $t \in [0,1]$, if $i'$ is as in \ref{i'}, then
\begin{align*}
\pi \circ \mathfrak{m}(b(i,t)) &= \Theta_1 \circ \Theta_2(b(i',t)) && \textrm{by \ref{i'}} \\
&= \Theta_1(\mathfrak{m}_T(b(i',t))) && \textrm{by choice of $\Theta_2$} \\
&= \mathfrak{m}_T(b(i',t)) && \textrm{since $\Theta_1$ is the identity on $A_{i'}$.}
\end{align*}
\end{proof}

\section{Combinatorial $n$-od covers}
\label{sec:comb covers}

In this section, we introduce the notion of a combinatorial $n$-od cover, which expresses some combinatorial data extracted from a small mesh $n$-od cover of a graph, embedded in the plane in the manner of the previous section.  This provides a tool for showing that certain plane continua are not simple $n$-od-like.

As above, let $n \geq 2$ be a fixed integer.  Throughout this section, assume that $G$ is a connected graph and $\omega \colon \mathsf{V}(G) \to \Gamma$ is a placement function.

\begin{defn}
\label{defn:comb cover}
Let $\delta > 0$.  A \emph{$\delta$-combinatorial $n$-od cover} for $\omega$ is a function $f \colon \mathsf{V}(G) \to \mathsf{V}(C)$ such that for any vertices $u,v,v_1,v_2,v_3$ of $G$ we have the following properties:
\begin{enumerate}[label={\normalfont \textbf{(C\Roman{*})}}]
\item \label{C1} If $f(u) = f(v)$, then either $\omega(u) = \omega(v) = o$, or there exists $i \in \{0,\ldots,n\}$ such that $\omega(u) = b(i,\cdot)$ and $\omega(v) = b(i,\cdot)$;
\item \label{C2} If $u$ and $v$ are adjacent in $G$, then $f(u)$ and $f(v)$ are adjacent in $C$; and
\item \label{C3} Suppose that $v_1,v_2,v_3$ are consecutive in $G$, $v \notin \{v_1,v_2,v_3\}$, $f(v_1) \neq f(v_3)$, $f(v) = f(v_2)$ and $f(v_2) \neq C(\cdot,0)$.  Suppose also that $0 \leq s < t \leq 1$, and that for some $i \in \{0,\ldots,n\}$, $\omega(v_2) = b(i,s)$ and $\omega(v) = b(i,t)$.  Then $t - s < \delta$.
\end{enumerate}
\end{defn}

This notion may be compared with that of a ``chain quasi-order'' from \cite[Section 3.1]{hoehn2011}.  Precisely, apart from the condition ``$f(v_2) \neq C(\cdot,0)$'' in \ref{C3}, a combinatorial $2$-od cover is otherwise equivalent to a chain quasi-order.

The following result is an analog of Proposition~5 of \cite{hoehn2011}.

\begin{prop}
\label{prop:cover to comb cover}
Let $\delta > 0$ be sufficiently small in the sense that:
\begin{itemize}
\item for any distinct $i_1,i_2 \in \{0,\ldots,n\}$ and any $t_1 \in [0,1]$, the distance between the point $b(i_1,t_1)$ and any point on the arc in $T_0$ from $o$ to $b(i_2,1)$ is greater than $4\delta$,
\end{itemize}
and let $\varepsilon_1,\varepsilon_2 > 0$ be such that $2\varepsilon_1 + \varepsilon_2 \leq \delta$.  Suppose $\Omega$ is a $\langle \omega,T_0,\varepsilon_1 \rangle$-embedding of $G$.  If there exists an open $n$-od cover of $\Omega(G)$ of mesh less than $\varepsilon_2$, then there exists a $\delta$-combinatorial $n$-od cover of $\omega$.
\end{prop}

\begin{proof}
Suppose $\mathcal{U}$ is an $n$-od cover of $\Omega(G)$ of mesh less than $\varepsilon_2$.  We refer to the elements of $\mathcal{U}$ using the notation $U(\ell,j)$, $\ell \in \{1,\ldots,n\}$, $j \geq 0$, adopting the same convention as for $n$-od graphs.  For each $i \in \{0,\ldots,n\}$, let $N(i)$ denote the open $2\delta$-neighborhood of $\{b(i,t): t \in [0,1]\}$ in $\mathbb{R}^2$.

Let $\mathcal{U}^{\mathsf{V}} = \{W \in \mathcal{U}: W \cap \Omega(\mathsf{V}(G)) \neq \emptyset\}$.  Define the function $g \colon \mathcal{U}^{\mathsf{V}} \to \{o\} \cup \{0,\ldots,n\}$ as follows.  Given $W \in \mathcal{U}^{\mathsf{V}}$, let $z_W \in \mathsf{V}(G)$ be such that $\Omega(z_W) \in W$.  If $\omega(z_W) = o$, then define $g(W) = o$; and if $\omega(z_W) = b(i,\cdot)$ for some $i \in \{0,\ldots,n\}$, then define $g(W) = i$.  Clearly $g$ is well-defined, since $\varepsilon_1$ and $\varepsilon_2$ are small compared with the distances between different regions in $T_0$.

\begin{claim}
\label{claim:chain on edge}
Let $u,v \in \mathsf{V}(G)$ be adjacent vertices in $G$, with $\omega(u) = o$ and $\omega(v) = b(i,\cdot)$ for some $i \in \{0,\ldots,n\}$.  Let $W_1,\ldots,W_m$ be a chain in $\mathcal{U}$ such that $W_j \cap \Omega(uv) \neq \emptyset$ for all $j \in \{1,\ldots,m\}$.

For each $j \in \{1,\ldots,m\}$, if $W_j \in \mathcal{U}^{\mathsf{V}}$, then $g(W_j) = o$ or $g(W_j) = i$.

Furthermore, suppose $j_1,j_2 \in \{1,\ldots,m\}$ are such that $W_{j_1},W_{j_2} \in \mathcal{U}^{\mathsf{V}}$, and $j_1 < j_2$.  Then:
\begin{enumerate}[label=(\alph{*})]
\item \label{both o} if $g(W_{j_1}) = g(W_{j_2}) = o$, then $W_j \cap N(i) = \emptyset$ for all $j_1 \leq j \leq j_2$; and
\item \label{both i} if $g(W_{j_1}) = g(W_{j_2}) = i$, then $W_j \cap N(i) \neq \emptyset$ for all $j_1 \leq j \leq j_2$.
\end{enumerate}
In particular, if $W_{j_3} \in \mathcal{U}^{\mathsf{V}}$ for some $j_3$ with $j_1 \leq j_3 \leq j_2$, then in case \ref{both o} we have $g(W_{j_3}) = o$, and in case \ref{both i} we have $g(W_{j_3}) = i$.
\end{claim}

\begin{proof}[Proof of Claim \ref{claim:chain on edge}]
\renewcommand{\qedsymbol}{\textsquare (Claim \ref{claim:chain on edge})}
For the first statement, let $W \in \mathcal{U}^{\mathsf{V}}$ be such that $g(W) = i'$ for some $i' \neq i$.  Let $z_W \in \mathsf{V}(G)$ be such that $z_W \in W$, so that $\omega(z_W) = b(i',t)$ for some $t \in [0,1]$.  By assumption on $\delta$, the distance from $b(i',t)$ to each point on the arc $\pi(\Omega(uv))$ is more than $4\delta$.  Also, the distance from $\Omega(z_W)$ to $\pi(\Omega(z_W))$ is less than $\varepsilon_1$, and for any $p \in uv$, the distance from $\Omega(p)$ to $\pi(\Omega(p))$ is less than $\varepsilon_1$.  Therefore, the distance from $\Omega(z_W)$ to $\Omega(p)$ is greater than $4\delta - 2\varepsilon_1 > \varepsilon_2$.  Since the mesh of $\mathcal{U}$ is less than $\varepsilon_2$, it follows that $\Omega(p) \notin W$, meaning that $W \cap \Omega(uv) = \emptyset$.

Next, we prove \ref{both o}.  Let $j_1,j_2 \in \{1,\ldots,m\}$ be such that $W_{j_1},W_{j_2} \in \mathcal{U}^{\mathsf{V}}$ and $j_1 < j_2$, and suppose $g(W_{j_1}) = g(W_{j_2}) = o$.  Suppose for a contradiction that there exists $j_1 \leq j \leq j_2$ such that $W_j \cap N(i) \neq \emptyset$.  Since the mesh of $\mathcal{U}$ is less than $\varepsilon_2 < \delta$ and by the smallness assumption on $\delta$, it is clear that $W_j \cap W_{j_1} = \emptyset = W_j \cap W_{j_2}$ (meaning in particular that $j_1 < j < j_2$).  Let $x_1 \in W_{j_1} \cap \Omega(uv)$, let $x_2 \in W_{j_2} \cap \Omega(uv)$, and let $y$ be a point in $W$ which is on the subarc of $\Omega(uv)$ from $x_1$ to $x_2$.  Also, let $z_1,z_2 \in \mathsf{V}(G)$ be such that $\Omega(z_1) \in W_{j_1}$ and $\Omega(z_2) \in W_{j_2}$, and note that $\omega(z_1) = \omega(z_2) = o$ since $g(W_{j_1}) = g(W_{j_2}) = o$.  Since the mesh of $\mathcal{U}$ is less than $\varepsilon_2$, and $\Omega$ is a $\langle \omega,T_0,\varepsilon_1 \rangle$-embedding of $G$, we have that the distance from $\pi(x_1)$ to $o$ is less than $2\varepsilon_1 + \varepsilon_2 \leq \delta$, and likewise for $\pi(x_2)$.  On the other hand, since $W_j \cap N(i) \neq \emptyset$, we similarly have that the distance from $\pi(y)$ to $N(i)$ is less than $\varepsilon_1 + \varepsilon_2 < \delta$.  By assumption on the smallness of $\delta$, we conclude that $\pi(y)$ cannot be on the arc in $T_0$ from $\pi(x_1)$ to $\pi(x_2)$.  But this contradicts the condition that $\pi$ is a homeomorphism from the arc $\Omega(uv)$ to $\pi(\Omega(uv))$.

The proof of \ref{both i} is similar.
\end{proof}

In the case that $g$ is not already defined on the branch element $U(\cdot,0)$ of $\mathcal{U}$ (i.e.\ if $U(\cdot,0) \notin \mathcal{U}^{\mathsf{V}}$), then we extend $g$ to $\mathcal{U}^{\mathsf{V}} \cup \{U(\cdot,0)\}$ by defining $g(U(\cdot,0)) = i$ if $U(\cdot,0) \cap N(i) \neq \emptyset$ for some $i \in \{0,\ldots,n\}$, and $g(U(\cdot,0)) = o$ otherwise.  Observe that this latter definition of $g(U(\cdot,0))$ concurs with the original definition of $g(U(\cdot,0))$ in the case that $U(\cdot,0) \in \mathcal{U}^{\mathsf{V}}$.

By an \emph{alternation} in $g$ in a given chain in $\mathcal{U}$, we mean a subchain $W_1,\ldots,W_m$ such that:
\begin{enumerate}
\item $W_1,W_m \in \mathcal{U}^{\mathsf{V}} \cup \{U(\cdot,0)\}$;
\item $W_j \notin \mathcal{U}^{\mathsf{V}} \cup \{U(\cdot,0)\}$ for each $1 < j < m$; and
\item $g(W_1) \neq g(W_m)$.
\end{enumerate}

Define $f \colon \mathsf{V}(G) \to C$ as follows.  Given $z \in \mathsf{V}(G)$, let $W^z = U(\ell_z,j_z) \in \mathcal{U}^{\mathsf{V}}$ be such that $\Omega(z) \in W^z$.  Define $f(z) = C(\ell_z,A_z)$, where $A_z$ is the number of alternations in $g$ in the chain in $\mathcal{U}$ from $U(\cdot,0)$ to $W^z$.  Clearly, this function $f$ is well-defined, i.e.\ does not depend on the choice of $W^z$ (in the case that $\Omega(z)$ is in the intersection of two elements of $\mathcal{U}$).  We now prove that $f$ satisfies properties \ref{C1}, \ref{C2}, and \ref{C3} of Definition~\ref{defn:comb cover}.

\medskip
\noindent \ref{C1}.  Let $u,v \in \mathsf{V}(G)$ be such that $f(u) = f(v)$.  Let $W^u,W^v \in \mathcal{U}^{\mathsf{V}}$ be such that $u \in W^u$ and $v \in W^v$.  It follows immediately from the definition of $f$ that there are no alternations in $g$ in the chain in $\mathcal{U}$ from $W^u$ to $W^v$.  Therefore, $g(u) = g(v)$, meaning that either $\omega(u) = \omega(v) = o$, or there exists $i \in \{0,\ldots,n\}$ such that $\omega(u) = b(i,\cdot)$ and $\omega(v) = b(i,\cdot)$.

\medskip
\noindent \ref{C2}.  Let $u,v \in \mathsf{V}(G)$ be adjacent vertices.  Assume without loss of generality that $\omega(u) = o$ and $\omega(v) = b(i,\cdot)$ for some $i \in \{0,\ldots,n\}$.  Let $W^u,W^v \in \mathcal{U}^{\mathsf{V}}$ be such that $u \in W^u$ and $v \in W^v$.  Let $W_1,\ldots,W_m$ be the chain in $\mathcal{U}$ from $W^u$ to $W^v$.

We consider two cases, depending on whether $U(\cdot,0)$ is in the chain $W_1,\ldots,W_m$ or not.

\medskip
\noindent \textbf{Case 1:} $U(\cdot,0)$ is not in the chain $W_1,\ldots,W_m$.

In this case, there exists $\ell \in \{1,\ldots,n\}$ such that $W^u$ and $W^v$ are both in leg $\ell$ of $\mathcal{U}$.  This means $f(u) = C(\ell,A_u)$ and $f(v) = C(\ell,A_v)$.

Clearly there is at least one alternation in $g$ in the chain $W_1,\ldots,W_m$, since $g(W^u) = o$ and $g(W^v) = i$.  On the other hand, since $W_j \cap \Omega(uv) \neq \emptyset$ for each $j \in \{1,\ldots,m\}$, it follows from Claim~\ref{claim:chain on edge} that there is at most one alternation in $g$ in the chain $W_1,\ldots,W_m$.  Therefore, the numbers $A_u$ and $A_v$ differ by exactly one, meaning that $f(u)$ and $f(v)$ are adjacent in $C$.

\medskip
\noindent \textbf{Case 2:} $U(\cdot,0)$ is in the chain $W_1,\ldots,W_m$.

Let $j_0 \in \{1,\ldots,m\}$ be such that $U(\cdot,0) \in W_{j_0}$.

If there exists $j \in \{1,\ldots,j_0\}$ such that $W_j \in \mathcal{U}^{\mathsf{V}}$ and $g(W_j) = i$, then by Claim~\ref{claim:chain on edge} applied to the chain $W_j,\ldots,W_m$ (which contains $U(\cdot,0) = W_{j_0}$), we have that $U(\cdot,0) \cap N(i) \neq \emptyset$.  By the definition of $g$, we then have $g(U(\cdot,0)) = i$, and there are no alternations in $g$ in the chain $W_j,\ldots,W_m$.  As in the previous case, there is exactly one alternation in the chain $W_1,\ldots,W_j$.  It follows that $A_u = 1$ and $A_v = 0$, hence $f(u)$ and $f(v)$ are adjacent in $C$.

If there exists $j \in \{j_0,\ldots,m\}$ such that $W_j \in \mathcal{U}^{\mathsf{V}}$ and $g(W_j) = o$, then similarly we may argue that $A_u = 0$ and $A_v = 1$, hence $f(u)$ and $f(v)$ are adjacent in $C$.

Finally, if neither of the previous two conditions holds, we immediately have that either $A_u = 0$ and $A_v = 1$ (if $g(U(\cdot,0)) = o$), or $A_u = 1$ and $A_v = 0$ (if $g(U(\cdot,0)) = i$).  In either case, $f(u)$ and $f(v)$ are adjacent in $C$.

\medskip
\noindent \ref{C3}.  Let $v_1,v_2,v_3 \in \mathsf{V}(G)$ be consecutive vertices in $G$, let $v \in \mathsf{V}(G) \smallsetminus \{v_1,v_2,v_3\}$, and suppose that $f(v_1) \neq f(v_3)$, $f(v) = f(v_2)$ and $f(v_2) \neq C(\cdot,0)$.  Suppose also that $0 \leq s < t \leq 1$, and that for some $i \in \{0,\ldots,n\}$, $\omega(v_2) = b(i,s)$ and $\omega(v) = b(i,t)$.

For $k = 1,2,3$, let $W^{v_k} \in \mathcal{U}^{\mathsf{V}}$ be such that $\Omega(v_k) \in W^{v_k}$.  Also, let $W^v \in \mathcal{U}^{\mathsf{V}}$ be such that $\Omega(v) \in W^v$.  By \ref{C2}, we have that $f(v_1),f(v_2),f(v_3)$ are consecutive vertices in $C$, and since $f(v_2) \neq C(\cdot,0)$, they are all in one leg of $C$.  It easily follows from the definition of $f$ that both $W^{v_2}$ and $W^v$ are in the chain in $\mathcal{U}$ from $W^{v_1}$ to $W^{v_3}$.

Each element of the chain in $\mathcal{U}$ from $W^{v_1}$ to $W^{v_2}$ intersects the arc $\Omega(v_1 v_2)$, and each element of the chain in $\mathcal{U}$ from $W^{v_2}$ to $W^{v_3}$ intersects the arc $\Omega(v_2 v_3)$.  In particular, $W^v$ intersects the arc $\Omega(v_1 v_2 \cup v_2 v_3)$.  Let $p$ be a point in $W^v \cap \Omega(v_1 v_2 \cup v_2 v_3)$.  According to the geometry of $T_0$ and the assumption about the smallness of $\delta$, the closest point of $\pi(\Omega(v_1 v_2 \cup v_2 v_3))$ to $\pi(\Omega(v)) = b(i,t)$ is $\pi(\Omega(v_2)) = b(i,s)$, and the distance between these points is $t - s$.  Now
\begin{align*}
t - s &\leq \| \pi(\Omega(v)) - \pi(\Omega(p)) \| \\
&\leq \| \pi(\Omega(v)) - \Omega(v) \| + \| \Omega(v) - \Omega(p) \| + \| \pi(\Omega(p)) - \Omega(p) \| \\
&< 2\varepsilon_1 + \|\Omega(v) - \Omega(p)\|  \qquad \textrm{since $\Omega$ is a $\langle \omega,T_0,\varepsilon_1 \rangle$-embedding} \\
&< 2\varepsilon_1 + \varepsilon_2  \qquad \textrm{since the mesh of $\mathcal{U}$ is less than $\varepsilon_2$} \\
&\leq \delta .
\end{align*}

This completes the proof that $f$ is a $\delta$-combinatorial $n$-od cover for $\omega$.
\end{proof}

\section{Wrapping patterns}
\label{sec:wrapping patterns}

In this section, we undertake a detailed combinatorial study of combinatorial $n$-od covers for placement functions having certain common patterns we utilize in our constructions later, in Section~\ref{sec:Ingram ex} (and future work).

As above, let $n \geq 2$ be a fixed integer.  Throughout this section, assume that $G$ is a connected graph, $\omega \colon \mathsf{V}(G) \to \Gamma$ is a placement function, $\delta > 0$, and $f \colon \mathsf{V}(G) \to \mathsf{V}(C)$ is a $\delta$-combinatorial $n$-od cover for $\omega$.

\begin{defn}
A \emph{wrapping pattern} in $G$ is a sequence of consecutive vertices $w_0,\ldots,w_k \in \mathsf{V}(G)$, where $k$ is a positive multiple of $2(n+1)$, such that:
\begin{enumerate}[label=(\arabic{*})]
\item for each even $p=2q \in \{0,\ldots,k\}$, $\omega(w_p) = b(q',\cdot)$, where $q'$ is equal to $q$ reduced modulo $n+1$; and
\item for each odd $p=2q+1 \in \{1,\ldots,k-1\}$, $\omega(w_p) = o$.
\end{enumerate}
\end{defn}

As an example of a wrapping pattern in the case $n = 3$, consider a sequence of consecutive vertices $w_0,\ldots,w_{16}$ whose images under $\omega$, in order, are
\begin{align*}
b(0,\tfrac{1}{2}),\ & o,\ b(1,\tfrac{3}{4}),\ o,\ b(2,0),\ o,\ b(3,1),\ o,\ b(0,\tfrac{2}{5}), \\
& o,\ b(1,\tfrac{1}{2}),\ o,\ b(2,\tfrac{1}{3}),\ o,\ b(3,\tfrac{1}{2}),\ o,\ b(0,1) .
\end{align*}

\begin{defn}
A sequence of consecutive vertices $v_0,\ldots,v_{2q} \in \mathsf{V}(G)$, where $q \in \{1,\ldots,n-1\}$, is a \emph{branch-star trajectory} if:
\begin{enumerate}[label=(\arabic{*})]
\item for each even $p \in \{0,\ldots,2q\}$, $f(v_p) = C(\ell_p,1)$ for some $\ell_p \in \{1,\ldots,n\}$;
\item for each odd $p \in \{0,\ldots,2q\}$, $f(v_p) = C(\cdot,0)$; and
\item these numbers $\ell_p$ are pairwise distinct.
\end{enumerate}
\end{defn}

\begin{prop}
\label{prop:L-covered}
Suppose $w_0,\ldots,w_k$ is a wrapping pattern in $G$.  Then either:
\begin{enumerate}[label=(\alph{*})]
\item \label{L away} (Goes away from branch on leg $\ell$) for some $\ell \in \{1,\ldots,n\}$ and $j \geq 0$, $f(w_p) = C(\ell,j+p)$ for each $p \in \{0,\ldots,k\}$; or
\item \label{L toward} (Goes toward branch on leg $\ell$) for some $\ell \in \{1,\ldots,n\}$ and $j \geq k$, $f(w_p) = C(\ell,j-p)$ for each $p \in \{0,\ldots,k\}$; or
\item \label{L across} (Goes across branch from leg $\ell_1$ to leg $\ell_2$) for some distinct $\ell_1,\ell_2 \in \{1,\ldots,n\}$ and some $j_1,j_2 \geq 1$ with $j_1 + j_2 \leq k$,
\begin{enumerate}[label=(\alph{enumi}\arabic{*})]
\item $f(w_p) = C(\ell_1,j_1-p)$ for each $p \in \{0,\ldots,j_1\}$;
\item $f(w_{k-p}) = C(\ell_2,j_2-p)$ for each $p \in \{0,\ldots,j_2\}$; and
\item $w_{j_1-1},\ldots,w_{k-j_2+1}$ is a branch-star trajectory.
\end{enumerate}
\end{enumerate}
\end{prop}

To illustrate the possibility \ref{L across} from Proposition~\ref{prop:L-covered}, consider the above example of a wrapping pattern $w_0,\ldots,w_{16}$ in the case $n = 3$.  If the images of $w_0,\ldots,w_{16}$ under $f$, in order, are
\begin{align*}
C(&3,5),\ C(3,4),\ C(3,3),\ C(3,2), \\
& C(3,1),\ C(\cdot,0),\ C(2,1),\ C(\cdot,0),\ C(1,1) \\
& C(1,2),\ C(1,3),\ C(1,4),\ C(1,5),\ C(1,6),\ C(1,7),\ C(1,8),\ C(1,9)
\end{align*}
then this wrapping pattern satisfies condition \ref{L across} with $\ell_1 = 3$, $\ell_2 = 1$, $j_1 = 5$, and $j_2 = 9$.  See Figure~\ref{fig:goes across} for a visual depiction of this wrapping pattern example.

\begin{figure}
\begin{center}
\includegraphics{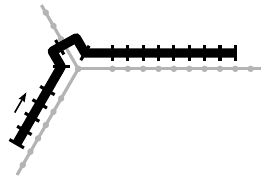}
\end{center}

\caption{Visual representation of a wrapping pattern example $w_0,\ldots,w_{16}$ in the case $n = 3$, which goes across the branch from leg $3$ (extending down and to the left) to leg $1$ (extending to the right).  The wrapping pattern is depicted as a track (direction given by the arrow), with tick marks indicating the locations of the images of the vertices, overlaid on the graph $C$, which is in grey.  Note that the spacing of vertices around the branch of $C$ is exaggerated, to make room for the turning track there.}
\label{fig:goes across}
\end{figure}

\begin{proof}
Note that by definition of a wrapping pattern,
\begin{enumerate}[label=(\arabic{section}.\arabic{thm}-\roman{*})]
\item \label{distinct} for each even $p \in \{2,\ldots,k-2\}$, $f(w_{p-2})$, $f(w_p)$, and $f(w_{p+2})$ must be three distinct vertices in $C$ (by \ref{C1}); and
\item \label{consecutive} for each even $p \in \{0,\ldots,k-2\}$, $f(w_p)$, $f(w_{p+1})$, and $f(w_{p+2})$ must be three consecutive vertices in $C$ (by \ref{C2}).
\end{enumerate}

Suppose first that, for some $\ell \in \{1,\ldots,n\}$ and $j \geq 0$, $f(w_0) = C(\ell,j)$ and $f(w_1) = C(\ell,j+1)$.  Applying \ref{distinct} and \ref{consecutive} repeatedly, it easily follows that $f(w_p) = C(\ell,j+p)$ for each $p \in \{0,\ldots,k\}$.  Thus we have alternative \ref{L away} (goes away from the branch on leg $\ell$) in this case.

Now suppose the opposite, in which case there exists $\ell_1 \in \{1,\ldots,n\}$ and $j_1 \geq 1$ such that $f(w_0) = C(\ell_1,j_1)$ and $f(w_1) = C(\ell_1,j_1-1)$.  By similar reasoning as above, we conclude that $f(w_p) = C(\ell_1,j_1-p)$ for each $p \in \{0,\ldots,\min\{j_1,k\}\}$.  In the event that $j_1 \geq k$, we thus have alternative \ref{L toward} (goes toward the branch on leg $\ell_1$) in this case.

Suppose, then, that $j_1 < k$.  Let $j_2 \geq 1$ be minimal such that $f(w_{k-j_2+1})$ is in the branch-star.  Note that since $f(w_{j_1+1})$ is in the branch-star, by minimality of $j_2$, we have $k - j_2 + 1 \geq j_1 + 1$, hence $j_1 + j_2 \leq k$.  If $j_1$ is even, then it follows from \ref{consecutive} that $f(w_{j_1+2})$ cannot be in the branch-star of $C$, and by \ref{distinct} it cannot be on leg $\ell_1$ of $C$ because $f(w_{j_1+2}) \neq f(w_{j_1-2})$.  This means that $j_2 = k - j_1$, and $w_{j_1-1},\ldots,w_{k-j_2+1} = w_{j_1-1},w_{j_1},w_{j_1+1}$ is a branch-star trajectory, with $f(w_{j_1-1}) = C(\ell_1,1)$, $f(w_{j_1}) = C(\cdot,0)$, and $f(w_{j_1+1}) = C(\ell_2,1)$ for some $\ell_2 \in \{1,\ldots,n\} \smallsetminus \{\ell_1\}$.

If $j_1$ is odd, then since $f(w_p)$ is in the branch-star for each even $p \in \{j_1-1,\ldots,k-j_2+1\}$, we have by \ref{consecutive} that $f(w_p) = C(\ell,1)$ for some $\ell \in \{1,\ldots,n\}$, and for each odd $p \in \{j_1,\ldots,k-j_2\}$, $f(w_p) = C(\cdot,0)$.  Moreover, we see that $j_2$ must be odd as well in this case, so $f(w_{k-j_2+1}) = C(\ell_2,1)$ for some $\ell_2 \in \{1,\ldots,n\}$.  It follows from \ref{C1} and the definition of a wrapping pattern, and from the pigeonhole principle, that there are at most $n$ even numbers $p$ in $\{j_1-1,\ldots,k-j_2+1\}$ (i.e.\ that $k - (j_1 + j_2) \leq 2n - 4$), and the vertices $f(w_p)$, $p \in \{j_1-1,\ldots,k-j_2+1\}$ even, are all distinct (in particular, $\ell_1 \neq \ell_2$).  Thus, $w_{j_1-1},\ldots,w_{k-j_2+1}$ is a branch-star trajectory.

By similar reasoning as above, we can show that $f(w_{k-p}) = C(\ell_2,j_2-p)$ for each $p \in \{0,\ldots,j_2\}$.  Thus, in this case, we have alternative \ref{L across} (goes across the branch from leg $\ell_1$ to leg $\ell_2$).
\end{proof}

Observe that in all cases from Proposition~\ref{prop:L-covered}, there exists a component $H_1$ of $C - f(w_0)$ such that $f(w_p) \in H_1$ for all $p \in \{1,\ldots,k\}$.  Likewise, there exists a component $H_2$ of $C - f(w_k)$ such that $f(w_p) \in H_2$ for all $p \in \{0,\ldots,k-1\}$.

\begin{defn}
Suppose $w_0,\ldots,w_k$ is a wrapping pattern in $G$.  We say that $w_0,\ldots,w_k$ \emph{meets the branch} if there exists $p \in \{0,\ldots,k\}$ such that $f(w_p) = C(\cdot,0)$.  If $w_0,\ldots,w_k$ meets the branch, we define the \emph{pre-branch-star segment} of $w_0,\ldots,w_k$ to be the (possibly empty) set $\{w_0,\ldots,w_{p-1}\}$, where $p \in \{0,\ldots,k\}$ is minimal such that $f(w_p)$ is in the branch-star.
\end{defn}

In the above example of a wrapping pattern $w_0,\ldots,w_{16}$ which goes across the branch, the pre-branch-star segment is $\{w_0,w_1,w_2,w_3\}$.

\subsection{Interaction between wrapping patterns}
\label{sec:sync wrapping patterns}

\begin{defn}
Let $\ell \in \{1,\ldots,n\}$ and $j \geq 2$.  The vertex $C(\ell,j)$ is called a \emph{reversal point} if there exist wrapping patterns $u_0,\ldots,u_{k_1}$ and $v_0,\ldots,v_{k_2}$ in $G$ such that either:
\begin{enumerate}[label=(\alph{*})]
\item (\emph{Start type}) $f(u_1) = C(\ell,j-1)$, $f(u_0) = f(v_0) = C(\ell,j)$, and $f(v_1) = C(\ell,j+1)$; or
\item (\emph{End type}) $f(u_{k_1-1}) = C(\ell,j-1)$, $f(u_{k_1}) = f(v_{k_2}) = C(\ell,j)$, and $f(v_{k_2-1}) = C(\ell,j+1)$.
\end{enumerate}

Equivalently, $C(\ell,j)$ is a start-type (respectively, end-type) reversal point if it is not in the branch-star of $C$, and there exist wrapping patterns $u_0,\ldots,u_{k_1}$ and $v_0,\ldots,v_{k_2}$ such that $f(u_0) = f(v_0) = C(\ell,j)$ (respectively, $f(u_{k_1}) = f(v_{k_2}) = C(\ell,j)$) and $f(u_1) \neq f(v_1)$ (respectively, $f(u_{k_1-1}) \neq f(v_{k_2-1})$).
\end{defn}

For an example of a reversal point, see Figure~\ref{fig:complex} below in Section~\ref{sec:wrapping complexes}; the vertex $C(1,11)$ (where leg $1$ is the one extending to the right) is a start-type reversal point, due to the wrapping patterns $\mathcal{Z}_{q_3}$ and $\mathcal{Z}_{q_4}$.

\begin{lem}
\label{lem:reversal pt}
Let $C(\ell,j)$ be a reversal point, where $\ell \in \{1,\ldots,n\}$ and $j \geq 2$, let $H_1$ and $H_2$ be the two components of $C - C(\ell,j)$, and let $w_0,\ldots,w_k$ be a wrapping pattern.
\begin{enumerate}[label=(\arabic{*})]
\item \label{no crossing} Either $f(w_p) \in H_1 \cup \{C(\ell,j)\}$ for all $p \in \{0,\ldots,k\}$, or $f(w_p) \in H_2 \cup \{C(\ell,j)\}$ for all $p \in \{0,\ldots,k\}$.
\item \label{hit reversal pt} Suppose $f(w_p) = C(\ell,j)$ for some $p \in \{0,\ldots,k\}$.  Then:
\begin{enumerate}[label=(\arabic{enumi}\alph{*})]
\item if $C(\ell,j)$ is a start-type reversal point, then $p = 0$; and
\item if $C(\ell,j)$ is an end-type reversal point, then $p = k$.
\end{enumerate}
\end{enumerate}
\end{lem}

\begin{proof}
We consider the case where $C(\ell,j)$ is a start-type reversal point; the case where $C(\ell,j)$ is an end-type reversal point is similar.  Let $u_0,\ldots,u_{k_1}$ and $v_0,\ldots,v_{k_2}$ be wrapping patterns such that $f(u_0) = f(v_0) = C(\ell,j)$ and $f(u_1) \neq f(v_1)$.

We first prove \ref{hit reversal pt}.  Suppose for a contradiction that $f(w_p) = C(\ell,j)$ for some $p \in \{1,\ldots,k\}$.  By \ref{C1} and the definition of a wrapping pattern, it follows that $p$ must be even (in fact, $p$ must be a multiple of $2(n+1)$ since $\omega(u_0) = b(0,\cdot)$).  By Proposition~\ref{prop:L-covered}, we must have either $f(w_{p-1}) = f(u_1)$ and $f(w_{p-2}) = f(u_2)$, or $f(w_{p-1}) = f(v_1)$ and $f(w_{p-2}) = f(v_2)$ (here we rely on the assumption that $j \geq 2$, so that these vertices are all on one leg of $C$).  Suppose, without loss of generality, that $f(w_{p-1}) = f(u_1)$ and $f(w_{p-2}) = f(u_2)$.  By definition of a wrapping pattern, $\omega(u_0) = b(0,\cdot)$ and $\omega(u_2) = b(1,\cdot)$.  So, by \ref{C1}, $\omega(w_p) = b(0,\cdot)$ and $\omega(w_{p-2}) = b(1,\cdot)$.  However, this contradicts the definition of a wrapping pattern, which entails that when $\omega(w_p) = b(0,\cdot)$ we have $\omega(w_{p-2}) = b(n,\cdot)$ (and $n > 1$).  This proves \ref{hit reversal pt}.

For \ref{no crossing}, we again argue by way of contradiction.  Suppose, for a contradiction, that there exist $p_1,p_2 \in \{0,\ldots,k\}$ such that $f(w_{p_1}) \in H_1 - C(\ell,j)$ and $f(w_{p_2}) \in H_2 - C(\ell,j)$.  We may assume, without loss of generality, that $p_1 < p_2$.  It follows from \ref{C2} (or Proposition~\ref{prop:L-covered}) that there exists $p$, with $p_1 < p < p_2$, such that $f(w_p) = f(u_0)$.  But this contradicts statement \ref{hit reversal pt}.
\end{proof}

\begin{lem}
\label{lem:absent branch-star}
There exists an integer $i^\ast \in \{0,\ldots,n\}$ such that for any $u \in \mathsf{V}(G)$ for which $f(u)$ is in the branch-star, we have $\omega(u) \neq b(i^\ast,\cdot)$.
\end{lem}

\begin{proof}
Given a vertex in the image of $f$, say $f(u)$ where $u \in \mathsf{V}(G)$, define $g(f(u)) = o$ if $\omega(u) = o$, and $g(f(u)) = i$ if $\omega(u) = b(i,\cdot)$ for some $i \in \{0,\ldots,n\}$.  This function $g \colon f(\mathsf{V}(G)) \to \{o\} \cup \{0,\ldots,n\}$ is well-defined by \ref{C1}.  We need to show that there exists $i^\ast \in \{0,\ldots,n\}$ which is not in the image of the branch-star under $g$.

In light of \ref{C2}, we may view $f$ as a graph homomorphism; that is, we may extend $f$ over the edges of $G$ by declaring that, for any edge $\{u,v\}$ in $G$, $f$ maps $\{u,v\}$ to the edge $\{f(u),f(v)\}$ in $C$.  Since $G$ is connected, the image of $G$ under $f$ in $C$ is connected.  Moreover, since $C$ has no cycles, it follows that for any two adjacent vertices $C(\ell,j)$, $C(\ell,j+1)$ in $C$ which are each in the image of $f$, there are adjacent vertices $u,v \in \mathsf{V}(G)$ such that $f(u) = C(\ell,j)$ and $f(v) = C(\ell,j+1)$; and, since $\omega$ is a placement function, one of $g(C(\ell,j))$ or $g(C(\ell,j+1))$ must be $o$.  Thus, not every vertex in the branch-star of $C$ can map into $\{0,\ldots,n\}$ under $g$.  Therefore, since there are only $n+1$ vertices in the branch-star, there exists $i^\ast \in \{0,\ldots,n\}$ which is not in the image of the branch-star under $g$.
\end{proof}

For the remainder of this paper, whenever a $\delta$-combinatorial $n$-od cover for a placement function is considered, we let $i^\ast$ denote a fixed integer satisfying the conclusion of Lemma~\ref{lem:absent branch-star}.

\begin{defn}
Given two sequences of vertices $u_0,\ldots,u_{k_1}$ and $v_0,\ldots,v_{k_2}$ in $G$, we will say that they have the \emph{same number of $b(i^\ast,\cdot)$'s} if the number of integers $p \in \{0,\ldots,k_1\}$ with $\omega(u_p) = b(i^\ast,\cdot)$ is equal to the number of integers $p \in \{0,\ldots,k_2\}$ with $\omega(v_p) = b(i^\ast,\cdot)$.
\end{defn}

\begin{lem}
\label{lem:wrapping sync start}
Let $u_0,\ldots,u_{k_1}$ and $v_0,\ldots,v_{k_2}$ be wrapping patterns, and suppose that either:
\begin{enumerate}[label=(\Alph{*})]
\item $f(u_0) = f(v_0)$; or
\item each of $u_0,\ldots,u_{k_1}$ and $v_0,\ldots,v_{k_2}$ meet the branch, and the pre-branch-star segments of $u_0,\ldots,u_{k_1}$ and $v_0,\ldots,v_{k_2}$ have the same number of $b(i^\ast,\cdot)$'s.
\end{enumerate}
For any $p_1 \in \{0,\ldots,k_1\}$ and $p_2 \in \{0,\ldots,k_2\}$, if $f(u_{p_1}) = f(v_{p_2})$, then either $\omega(u_{p_1}) = \omega(v_{p_2}) = o$ and $f(u_{p_1}) = C(\cdot,0)$, or $p_1 = p_2$.
\end{lem}

\begin{proof}
In the case that $f(u_0) = f(v_0)$ and at least one of $u_0,\ldots,u_{k_1}$ or $v_0,\ldots,v_{k_2}$ does not go across the branch to another leg of $C$, then the conclusion follows readily from Proposition~\ref{prop:L-covered}.  If $f(u_0) = f(v_0)$ and both $u_0,\ldots,u_{k_1}$ and $v_0,\ldots,v_{k_2}$ go across the branch to another leg of $C$, then the pre-branch-star segments of $u_0,\ldots,u_{k_1}$ and $v_0,\ldots,v_{k_2}$ are identical, hence have the same number of $b(i^\ast,\cdot)$'s.  Thus, it suffices to establish the conclusion of the Lemma in the case that $u_0,\ldots,u_{k_1}$ goes across the branch from some leg $\ell_1(u)$ to another leg $\ell_2(u)$, and $v_0,\ldots,v_{k_2}$ goes across the branch from some leg $\ell_1(v)$ to another leg $\ell_2(v)$, and the pre-branch-star segments of $u_0,\ldots,u_{k_1}$ and $v_0,\ldots,v_{k_2}$ have the same number of $b(i^\ast,\cdot)$'s.  Let $r \geq 0$ denote this number of $b(i^\ast,\cdot)$'s.  Let $j_1(u),j_2(u)$ (respectively, $j_1(v),j_2(v)$) be as in Proposition~\ref{prop:L-covered}\ref{L across} for the wrapping pattern $u_0,\ldots,u_{k_1}$ (respectively, $v_0,\ldots,v_{k_2}$).

We begin with several preliminary observations.  First, we establish bounds for the numbers $j_1(u)$, $j_1(v)$, $k_1 - j_1(u)$, and $k_2 - j_2(v)$, in terms of $r$.

The indices $p$ for which $\omega(u_p) = b(i^\ast,\cdot)$ are those which are congruent to $2i^\ast$ modulo $2(n+1)$.  The vertices mapped by $f$ to the branch-star are $u_p$ for $p \in \{j_1(u) - 1, \ldots, k_1 - j_2(u) + 1\}$, and for each of these, $\omega(u_p) \neq b(i^\ast,\cdot)$ (by definition of $i^\ast$).  Since $r$ is the number of $b(i^\ast,\cdot)$'s among the vertices $u_0,\ldots,u_{j_1(u)-2}$, we conclude that $2i^\ast + 2(r-1)(n+1) \leq j_1(u) - 2$, and $2i^\ast + 2r(n+1) \geq k_1 - j_2(u) + 2$.  Combining these with the inequality $j_1(u) + j_2(u) \leq k_1$, and making similar observations for the wrapping pattern $v_0,\ldots,v_{k_2}$, we have:
\begin{enumerate}[label=(\arabic{section}.\arabic{thm}-\roman{*})]
\item \label{j ineq 1} $2i^\ast + 2(r-1)(n+1) + 2 \leq j_1(u) \leq k_1 - j_2(u) \leq 2i^\ast + 2r(n+1) - 2$; and

\noindent
$2i^\ast + 2(r-1)(n+1) + 2 \leq j_1(v) \leq k_2 - j_2(v) \leq 2i^\ast + 2r(n+1) - 2$.
\end{enumerate}
Consequently,
\begin{enumerate}[resume, label=(\arabic{section}.\arabic{thm}-\roman{*})]
\item \label{j ineq 2} $|j_1(u) - j_1(v)| \leq 2(n+1) - 4$; and

\noindent
$|(k_1 - j_2(u)) - (k_2 - j_2(v))| \leq 2(n+1) - 4$.
\end{enumerate}

Our next observation is that when two wrapping patterns intersect in $C - C(\cdot,0)$ at an odd index, they must be going in the same direction.  Specifically:
\begin{enumerate}[resume, label=(\arabic{section}.\arabic{thm}-\roman{*})]
\item \label{odd direction match} Suppose $p_1 \in \{1,\ldots,k_1-1\}$ and $p_2 \in \{1,\ldots,k_2-1\}$ are odd (so that $\omega(u_{p_1}) = \omega(v_{p_2}) = o$) and $f(u_{p_1}) = f(v_{p_2}) \neq C(\cdot,0)$.  Then $f(u_{p_1-1}) = f(v_{p_2-1})$ and $f(u_{p_1+1}) = f(v_{p_2+1})$.
\end{enumerate}
Indeed, suppose for a contradiction that this does not hold.  By Proposition~\ref{prop:L-covered}, the only other possibility is $f(u_{p_1-1}) = f(v_{p_2+1})$ and $f(u_{p_1+1}) = f(v_{p_2-1})$.  By \ref{C1}, this means that, for some $i,i' \in \{0,\ldots,n\}$, we have $\omega(u_{p_1-1}) = b(i,\cdot)$ and $\omega(v_{p_2+1}) = b(i,\cdot)$, and $\omega(u_{p_1+1}) = b(i',\cdot)$ and $\omega(v_{p_2-1}) = b(i',\cdot)$.  By definition of a wrapping pattern, this would mean that both $i' = i+1$ and $i' = i-1$ (modulo $n+1$), a contradiction (since $n > 1$).

Next, let $p_1 \in \{0,\ldots,k_1\}$ and $p_2 \in \{0,\ldots,k_2\}$ be such that $|p_1 - p_2| < 2(n+1)$.  It follows from the definition of a wrapping pattern that if $p_1$ and $p_2$ are even, and $\omega(u_{p_1}) = b(i,\cdot)$ and $\omega(v_{p_2}) = b(i,\cdot)$, for some $i \in \{0,\ldots,n\}$ (in particular, if $f(u_{p_1}) = f(v_{p_2})$), then $p_1 = p_2$.  Suppose, now, that $p_1$ and $p_2$ are odd, and that $f(u_{p_1}) = f(v_{p_2}) \neq C(\cdot,0)$.  Combining \ref{odd direction match} and the previous observation applied to $p_1-1$ and $p_2-1$, which are even, we deduce again that $p_1 = p_2$.  Altogether, we have:
\begin{enumerate}[resume, label=(\arabic{section}.\arabic{thm}-\roman{*})]
\item \label{close match} Suppose $p_1 \in \{0,\ldots,k_1\}$ and $p_2 \in \{0,\ldots,k_2\}$ are such that $|p_1 - p_2| < 2(n+1)$ and $f(u_{p_1}) = f(v_{p_2})$.  Then either $\omega(u_{p_1}) = \omega(v_{p_2}) = o$ and $f(u_{p_1}) = C(\cdot,0)$, or $p_1 = p_2$.
\end{enumerate}

We are now ready to establish the conclusion of the Lemma.  Let $p_1 \in \{0,\ldots,k_1\}$ and $p_2 \in \{0,\ldots,k_2\}$, and suppose $f(u_{p_1}) = f(v_{p_2})$.  There are three cases to consider, and by \ref{close match} it suffices to prove that $|p_1 - p_2| < 2(n+1)$ in each case.

\medskip
\noindent \textbf{Case 1:} $f(u_{p_1})$ is on leg $\ell_1(u)$ and not in the branch-star, meaning $0 \leq p_1 \leq j_1(u) - 2$.

Let $\ell = \ell_1(u)$.  If $p_2 \geq k_2 - j_2(v) + 2$, so that $\ell_2(v) = \ell$, then by Proposition~\ref{prop:L-covered} we have $f(u_{p_1+1}) = f(v_{p_2-1})$, and one easily obtains a contradiction with \ref{odd direction match}.  Therefore, we must have $0 \leq p_2 \leq j_1(v) - 2$, and $\ell_1(v) = \ell$.

Now $f(u_{p_1}) = C(\ell,j_1(u) - p_1)$ and $f(v_{p_2}) = C(\ell,j_1(v) - p_2)$, so $j_1(u) - p_1 = j_1(v) - p_2$.  Therefore, by \ref{j ineq 2},
\[ |p_1 - p_2| = |j_1(u) - j_1(v)| < 2(n+1) .\]

\medskip
\noindent \textbf{Case 2:} $f(u_{p_1})$ is on leg $\ell_2(u)$ and not in the branch-star, meaning $k_1 - j_2(u) + 2 \leq p_1 \leq k_1$.

Let $\ell' = \ell_2(u)$.  Similar to the previous case, if $p_2 \leq j_1(v) - 2$, so that $\ell_1(v) = \ell'$, then by Proposition~\ref{prop:L-covered} we have $f(u_{p_1-1}) = f(v_{p_2+1})$, and one easily obtains a contradiction with \ref{odd direction match}.  Therefore, we must have $k_2 - j_2(v) + 2 \leq p_2 \leq k_2$, and $\ell_2(v) = \ell'$.

Now $f(u_{p_1}) = C(\ell',j_2(u) - k_1 + p_1)$ and $f(v_{p_2}) = C(\ell',j_2(v) - k_2 + p_2)$, so $j_2(u) - k_1 + p_1 = j_2(v) - k_2 + p_2$.  Therefore, by \ref{j ineq 2},
\[ |p_1 - p_2| = |(k_1 - j_2(u)) - (k_2 - j_2(v))| < 2(n+1) .\]

\medskip
\noindent \textbf{Case 3:} $f(u_{p_1})$ is in the branch-star, meaning $j_1(u) - 1 \leq p_1 \leq k_1 - j_2(u) + 1$, and $j_1(v) - 1 \leq p_2 \leq k_2 - j_2(v) + 1$.

Combining these inequalities with \ref{j ineq 1}, we have:
\[ 2i^\ast + 2(r-1)(n+1) + 1 \leq p_1,p_2 \leq 2i^\ast + 2r(n+1) - 1 ,\]
which means that
\[ |p_1 - p_2| \leq 2(n+1) - 2 < 2(n+1) .\]
\end{proof}

\begin{lem}
\label{lem:wrapping sync end}
Let $u_0,\ldots,u_k$ and $v_0,\ldots,v_k$ be wrapping patterns of the same length, and suppose that $f(u_k) = f(v_k)$.
\begin{enumerate}[label=(\arabic{*})]
\item \label{sync end 1} For any $p_1,p_2 \in \{0,\ldots,k\}$, if $f(u_{p_1}) = f(v_{p_2})$, then either $\omega(u_{p_1}) = \omega(v_{p_2}) = o$ and $f(u_{p_1}) = C(\cdot,0)$, or $p_1 = p_2$.
\item \label{sync end 2} If each of $u_0,\ldots,u_k$ and $v_0,\ldots,v_k$ meet the branch, then the pre-branch-star segments of $u_0,\ldots,u_k$ and $v_0,\ldots,v_k$ have the same number of $b(i^\ast,\cdot)$'s.
\end{enumerate}
\end{lem}

\begin{proof}
As with the proof of Lemma~\ref{lem:wrapping sync start}, if at least one of $u_0,\ldots,u_k$ or $v_0,\ldots,v_k$ does not go across the branch to another leg of $C$, then \ref{sync end 1} and \ref{sync end 2} follow readily from Proposition~\ref{prop:L-covered}.  Suppose, then, that $u_0,\ldots,u_k$ goes across the branch from some leg $\ell_1(u)$ to another leg $\ell_2(u)$, and $v_0,\ldots,v_k$ goes across the branch from some leg $\ell_1(v)$ to another leg $\ell_2(v)$.  Let $j_1(u),j_2(u)$ (respectively, $j_1(v),j_2(v)$) be as in Proposition~\ref{prop:L-covered}\ref{L across} for the wrapping pattern $u_0,\ldots,u_k$ (respectively, $v_0,\ldots,v_k$).

Since $f(u_k) = f(v_k)$, we see that $\ell_2(u) = \ell_2(v)$ and $j_2(u) = j_2(v)$, and in fact (by Proposition~\ref{prop:L-covered}) $f(u_{k-p}) = f(v_{k-p})$ for each $p \in \{0,\ldots,j_2(u)\}$.  Since the wrapping patterns $u_0,\ldots,u_k$ and $v_0,\ldots,v_k$ have the same length, it follows that the pre-branch-star segments of $u_0,\ldots,u_k$ and $v_0,\ldots,v_k$ have the same number of $b(i^\ast,\cdot)$'s.  This proves statement \ref{sync end 2}.  Statement \ref{sync end 1} now follows from Lemma~\ref{lem:wrapping sync start}.
\end{proof}

\subsection{Wrapping complexes}
\label{sec:wrapping complexes}

The placement functions we will encounter in Section~\ref{sec:Ingram ex} below will contain several wrapping patterns which are connected to one another in a way that is captured by the following Definition.

\begin{defn}
\label{defn:complex}
A \emph{wrapping complex} in $G$, under $f$, is a collection
\[ \left\{ \mathcal{Z}_q = \{z_0^{(q)},\ldots,z_{k(q)}^{(q)}\}: q \in \mathsf{V}(Q) \right\} \]
of wrapping patterns in $G$, indexed by the vertices of a connected graph $Q$, such that there is an edge between $q,q' \in \mathsf{V}(Q)$ if and only if either:
\begin{enumerate}[label=(\alph{*})]
\item $f \left( z_0^{(q)} \right) = f \left( z_0^{(q')} \right)$, or
\item $f \left( z_{k(q)}^{(q)} \right) = f \left( z_{k(q')}^{(q')} \right)$ and $k(q) = k(q')$.
\end{enumerate}
\end{defn}

An example of a wrapping complex is depicted in Figure~\ref{fig:complex}.

\begin{figure}
\begin{center}
\includegraphics{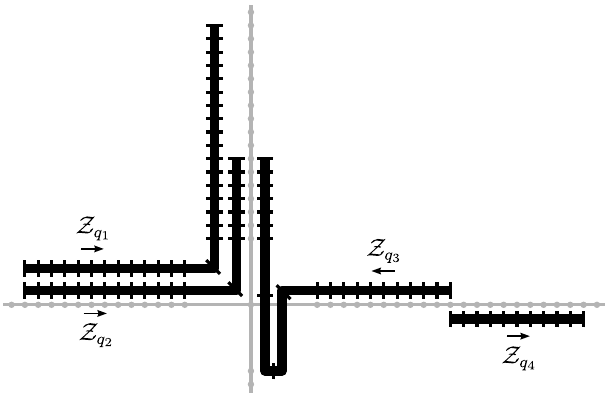}
\end{center}

\caption{Visual representation of a wrapping complex example with four wrapping patterns $\mathcal{Z}_{q_1},\mathcal{Z}_{q_2},\mathcal{Z}_{q_3},\mathcal{Z}_{q_4}$, in the case $n = 4$.  Here $k(q_1) = 30$, $k(q_2) = k(q_3) = 20$, and $k(q_4) = 10$.  We use the same conventions as in Figure~\ref{fig:goes across} to indicate the directions of the wrapping patterns and the locations of the images of the vertices.  Note that the spacing of vertices around the branch of $C$ is exaggerated, to make room for the turning tracks there.}
\label{fig:complex}
\end{figure}

\begin{prop}
\label{prop:complex sync}
Let $\left\{ \mathcal{Z}_q = \{z_0^{(q)},\ldots,z_{k(q)}^{(q)}\}: q \in \mathsf{V}(Q) \right\}$ be a wrapping complex in $G$, and let $q_1,q_2 \in \mathsf{V}(Q)$.  Then:
\begin{enumerate}[label=(\arabic{*})]
\item \label{complex sync 1} For any $p_1 \in \{0,\ldots,k(q_1)\}$ and $p_2 \in \{0,\ldots,k(q_2)\}$, if $f \left( z_{p_1}^{(q_1)} \right) = f \left( z_{p_2}^{(q_2)} \right)$, then either $\omega \left( z_{p_1}^{(q_1)} \right) = \omega \left( z_{p_2}^{(q_2)} \right) = o$ and $f \left( z_{p_1}^{(q_1)} \right) = C(\cdot,0)$, or $p_1 = p_2$.
\item \label{complex sync 2} If each of $\mathcal{Z}_{q_1}$ and $\mathcal{Z}_{q_2}$ meet the branch, then the pre-branch-star segments of $\mathcal{Z}_{q_1}$ and of $\mathcal{Z}_{q_2}$ have the same number of $b(i^\ast,\cdot)$'s.
\end{enumerate}
Furthermore, suppose $C(\ell,j)$ is a reversal point, where $\ell \in \{1,\ldots,n\}$ and $j \geq 2$, and let $H_1$ and $H_2$ be the two components of $C - C(\ell,j)$.  Then:
\begin{enumerate}[resume, label=(\arabic{*})]
\item \label{complex sync 3} If $Q_1$ (respectively, $Q_2$) is the subgraph of $Q$ generated by the set of all vertices $q \in \mathsf{V}(Q)$ such that $f(\mathcal{Z}_q) \subset H_1 \cup \{C(\ell,j)\}$ (respectively, $f(\mathcal{Z}_q) \subset H_2 \cup \{C(\ell,j)\}$), then $\mathsf{V}(Q) = \mathsf{V}(Q_1) \cup \mathsf{V}(Q_2)$, and $\{\mathcal{Z}_q: q \in \mathsf{V}(Q_1)\}$ and $\{\mathcal{Z}_q: q \in \mathsf{V}(Q_2)\}$ are wrapping complexes.
\end{enumerate}
\end{prop}

\begin{proof}
We proceed by induction on $|\mathsf{V}(Q)|$.  In the case that $|\mathsf{V}(Q)| = 1$, \ref{complex sync 1} follows from Proposition~\ref{prop:L-covered}, \ref{complex sync 2} is vacuous, and \ref{complex sync 3} follows from Lemma~\ref{lem:reversal pt}\ref{no crossing}.

Suppose now that $|\mathsf{V}(Q)| > 1$, and suppose for induction that \ref{complex sync 1}, \ref{complex sync 2}, and \ref{complex sync 3}  hold for any wrapping complex with fewer than $|\mathsf{V}(Q)|$ wrapping patterns.

We first argue \ref{complex sync 3}.  That $\mathsf{V}(Q) = \mathsf{V}(Q_1) \cup \mathsf{V}(Q_2)$ is clear from Lemma~\ref{lem:reversal pt}\ref{no crossing}.  To prove that $\{\mathcal{Z}_q: q \in \mathsf{V}(Q_1)\}$ and $\{\mathcal{Z}_q: q \in \mathsf{V}(Q_2)\}$ are wrapping complexes, we need only show that the subgraphs $Q_1$ and $Q_2$ are connected.  Clearly, we may assume $Q_1$ and $Q_2$ are both non-empty.  Since $Q$ is connected, there must exist $q_1 \in \mathsf{V}(Q_1)$ and $q_2 \in \mathsf{V}(Q_2)$ which are joined by an edge in $Q$.  We consider two cases, depending on the type of the reversal point.

First, suppose that $C(\ell,j)$ is a start-type reversal point.  In this case, by Lemma~\ref{lem:reversal pt}\ref{hit reversal pt} we must have that $f \left( z_0^{(q_1)} \right) = f \left( z_0^{(q_2)} \right) = C(\ell,j)$, and in fact, if $q \in \mathsf{V}(Q_1)$ is any vertex which is connected to a vertex in $Q_2$ by an edge in $Q$, then $f \left( z_0^{(q)} \right) = C(\ell,j)$, which means that $q$ and $q_1$ are joined by an edge in $Q_1$.  Consequently, the graph $Q_1$ is obtained from $Q$ by contracting the subgraph $Q_2$, together with the edge $\{q_1,q_2\}$, to the vertex $q_1$, hence it is connected.  Similarly, we may argue that $Q_2$ is connected.

Now suppose that $C(\ell,j)$ is an end-type reversal point.  In this case, by Lemma~\ref{lem:reversal pt}\ref{hit reversal pt} we must have that $f \left( z_{k(q_1)}^{(q_1)} \right) = f \left( z_{k(q_2)}^{(q_2)} \right) = C(\ell,j)$ and $k(q_1) = k(q_2)$.  Let $k$ denote this number $k(q_1)$.  Let $q \in \mathsf{V}(Q_1)$ be any vertex which is connected to a vertex in $Q_2$ by an edge in $Q$, which again, by Lemma~\ref{lem:reversal pt}\ref{hit reversal pt}, means that $f \left( z_{k(q)}^{(q)} \right) = C(\ell,j)$.  Choose a vertex $q_0 \in \mathsf{V}(Q) \smallsetminus \{q\}$ for which the subgraph $Q - q_0$ is connected.  Note that at least one of $q_1$ or $q_2$ is in $\mathsf{V}(Q - q_0)$.  Therefore, since $f \left( z_{k(q)}^{(q)} \right) = f \left( z_{k(q_1)}^{(q_1)} \right) = f \left( z_{k(q_2)}^{(q_2)} \right)$, we conclude by inductive hypothesis \ref{complex sync 1} that $k(q) = k$.  Therefore, $q$ and $q_1$ are joined by an edge in $Q_1$.  Consequently, as in the previous case, the graph $Q_1$ is obtained from $Q$ by contracting the subgraph $Q_2$, together with the edge $\{q_1,q_2\}$, to the vertex $q_1$, hence it is connected.  Similarly, we may argue that $Q_2$ is connected.  This completes the proof of \ref{complex sync 3}.

For \ref{complex sync 1} and \ref{complex sync 2}, we consider two cases.

\medskip
\noindent \textbf{Case 1:} Each wrapping pattern $\mathcal{Z}_q$, $q \in \mathsf{V}(Q)$, meets the branch.

Let $q_0 \in \mathsf{V}(Q)$ be such that $Q - q_0$ is connected, and let $q' \in \mathsf{V}(Q) \smallsetminus \{q_0\}$ be a vertex which is connected by an edge to $q_0$ in $Q$.  If $f \left( z_0^{(q_0)} \right) = f \left( z_0^{(q')} \right)$, then clearly the pre-branch-star segments of $\mathcal{Z}_{q_0}$ and $\mathcal{Z}_{q'}$ coincide, hence both have the same number of $b(i^\ast,\cdot)$'s.  If $f \left( z_{k(q_0)}^{(q_0)} \right) = f \left( z_{k(q')}^{(q')} \right)$ and $k(q_0) = k(q')$, then the pre-branch-star segments of $\mathcal{Z}_{q_0}$ and $\mathcal{Z}_{q'}$ have the same number of $b(i^\ast,\cdot)$'s by Lemma~\ref{lem:wrapping sync end}\ref{sync end 2}.  Since \ref{complex sync 2} holds for the wrapping complex $\{\mathcal{Z}_q: q \in \mathsf{V}(Q - q_0)\}$ by induction, it follows that \ref{complex sync 2} holds for $\{\mathcal{Z}_q: q \in \mathsf{V}(Q)\}$ as well.  Conclusion \ref{complex sync 1} now follows from \ref{complex sync 2} and Lemma~\ref{lem:wrapping sync start}.

\medskip
\noindent \textbf{Case 2:} There is at least one $q \in \mathsf{V}(Q)$ such that $\mathcal{Z}_q$ does not meet the branch.

Choose $q^\ast \in \mathsf{V}(Q)$ to be such an element $q$ for which $k(q)$ is minimal.  Let $q' \in \mathsf{V}(Q) \smallsetminus \{q^\ast\}$ be a vertex which is connected by an edge in $Q$ to $q^\ast$.  We must consider two subcases, depending on the nature of this edge.

\medskip
\noindent \textbf{Subcase 2.1:} $f \left( z_0^{(q^\ast)} \right) = f \left( z_0^{(q')} \right)$.

If $f \left( z_1^{(q^\ast)} \right) \neq f \left( z_1^{(q')} \right)$, then $f \left( z_0^{(q^\ast)} \right)$ is a start-type reversal point.  By \ref{complex sync 3}, $\{\mathcal{Z}_q: q \in \mathsf{V}(Q)\}$ splits into two wrapping complexes, each with fewer than $|\mathsf{V}(Q)|$ wrapping patterns.  Therefore, by induction applied to each, and since the only vertex in common between their images under $f$ is $f \left( z_0^{(q^\ast)} \right)$, we have that \ref{complex sync 1} and \ref{complex sync 2} hold.

Suppose, then, that $f \left( z_1^{(q^\ast)} \right) = f \left( z_1^{(q')} \right)$.  By minimality of $k(q^\ast)$, it follows that $k(q') \geq k(q^\ast)$, and by Proposition~\ref{prop:L-covered},
\begin{enumerate}[label=($\ast$)]
\item \label{qast q' match} $f \left( z_p^{(q^\ast)} \right) = f \left( z_p^{(q')} \right)$ for each $p \in \{0,\ldots,k(q^\ast)\}$.
\end{enumerate}
We claim that $\{\mathcal{Z}_q: q \in \mathsf{V}(Q - q^\ast)\}$ is a wrapping complex.  Indeed, let $q \in \mathsf{V}(Q) \smallsetminus \{q^\ast\}$ be a vertex which is connected by an edge to $q^\ast$ in $Q$.  If $f \left( z_0^{(q)} \right) = f \left( z_0^{(q^\ast)} \right)$, then immediately we have $f \left( z_0^{(q)} \right) = f \left( z_0^{(q')} \right)$.  Suppose $f \left( z_{k(q)}^{(q)} \right) = f \left( z_{k(q^\ast)}^{(q^\ast)} \right)$ and $k(q) = k(q^\ast)$.  If $k(q') = k(q^\ast)$, then $f \left( z_{k(q)}^{(q)} \right) = f \left( z_{k(q')}^{(q')} \right)$ and $k(q) = k(q')$.  If $k(q') > k(q^\ast)$, then we must have $f \left( z_{k(q)-1}^{(q)} \right) = f \left( z_{k(q^\ast)-1}^{(q^\ast)} \right)$, for if not, then $f \left( z_{k(q^\ast)}^{(q^\ast)} \right)$ would be a reversal point, and $f \left( z_{k(q^\ast)}^{(q')} \right) = f \left( z_{k(q^\ast)}^{(q^\ast)} \right)$ by \ref{qast q' match}, a contradiction with Lemma~\ref{lem:reversal pt}\ref{hit reversal pt}.  It then follows from Proposition~\ref{prop:L-covered} that $f \left( z_0^{(q)} \right) = f \left( z_0^{(q^\ast)} \right) = f \left( z_0^{(q')} \right)$.  In any case, $q$ and $q'$ are connected by an edge in $Q - q^\ast$.  Therefore, $Q - q^\ast$ is connected, hence $\{\mathcal{Z}_q: q \in \mathsf{V}(Q - q^\ast)\}$ is a wrapping complex.  We now have \ref{complex sync 1} and \ref{complex sync 2} by induction and by \ref{qast q' match}.

\medskip
\noindent \textbf{Subcase 2.2:} $f \left( z_{k(q^\ast)}^{(q^\ast)} \right) = f \left( z_{k(q')}^{(q')} \right)$ and $k(q^\ast) = k(q')$.

If $f \left( z_{k(q^\ast)-1}^{(q^\ast)} \right) \neq f \left( z_{k(q')-1}^{(q')} \right)$, then $f \left( z_{k(q^\ast)}^{(q^\ast)} \right)$ is an end-type reversal point.  By \ref{complex sync 3}, $\{\mathcal{Z}_q: q \in \mathsf{V}(Q)\}$ splits into two wrapping complexes, each with fewer than $|\mathsf{V}(Q)|$ wrapping patterns.  Therefore, by induction applied to each, and since the only vertex in common between their images under $f$ is $f \left( z_{k(q^\ast)}^{(q^\ast)} \right)$, we have that \ref{complex sync 1} and \ref{complex sync 2} hold.

Suppose, then, that $f \left( z_{k(q^\ast)-1}^{(q^\ast)} \right) = f \left( z_{k(q')-1}^{(q')} \right)$.  Since $k(q^\ast) = k(q')$ and $\mathcal{Z}_{q^\ast}$ does not meet the branch, it follows from Proposition~\ref{prop:L-covered} that $f \left( z_p^{(q^\ast)} \right) = f \left( z_p^{(q')} \right)$ for each $p \in \{0,\ldots,k(q^\ast)\}$, i.e.\ the images of $\mathcal{Z}_{q^\ast}$ and $\mathcal{Z}_{q'}$ under $f$ are identical.  Clearly, then, $\{\mathcal{Z}_q: q \in \mathsf{V}(Q - q^\ast)\}$ is a wrapping complex, and \ref{complex sync 1} and \ref{complex sync 2} hold by induction.
\end{proof}

\section{Placement function expansion and composition of projections}
\label{sec:expansion}

To construct complex continua in the plane, our strategy will be to produce a sequence of $\Gamma$-spaces, starting with $T_0$, each one described in terms of the previous one by a placement function.  The desired continuum will be the Hausdorff limit of the sequence of graphs.  But, in order to apply Proposition~\ref{prop:cover to comb cover} to rule out the existence of a small mesh $n$-od cover of this continuum, we will need to have a description of each graph in the sequence in terms of a placement function relative to $T_0$.  To facilitate this, we develop the notion of placement function ``expansion'' in this section.

As above, let $n \geq 2$ be a fixed integer.  Throughout this section, assume that $G$ is a graph, and $\omega \colon \mathsf{V}(G) \to \Gamma$ is a placement function.

\begin{defn}
\label{defn:subdivision}
A \emph{subdivision} of $G$ is a graph $G^+$ such that:
\begin{enumerate}[label=(\arabic{*})]
\item $\mathsf{V}(G) \subseteq \mathsf{V}(G^+)$;
\item \label{edge subdivision} For each edge $\{u,v\}$ in $G$, there exists a unique sequence of consecutive vertices $z_0,\ldots,z_\kappa \in \mathsf{V}(G^+)$ in $G^+$ with $z_0 = u$ and $z_\kappa = v$; and
\item All of the vertices and edges of $G^+$ are accounted for in \ref{edge subdivision}.
\end{enumerate}
\end{defn}

For the remainder of this section, let $\chi \colon \mathsf{V}(G_\chi) \to \Gamma$ be a placement function for an $(n+1)$-od graph $G_\chi$, such that:
\begin{enumerate}[label=(\arabic{section}-\roman{*})]
\item \label{chi1} $\chi(G_\chi(\cdot,0)) \neq o$, where $G_\chi(\cdot,0)$ is the branch vertex of $G_\chi$; and
\item \label{chi2} for each endpoint $e$ of $G_\chi$, $\chi(e) = b(i(e),1)$ for some $i(e) \in \{0,\ldots,n\}$.
\end{enumerate}

The following Definition is an adaptation of the notion of a ``$\rho_N$-expansion'' from \cite[Section 4.2]{hoehn2011}, and Proposition~\ref{prop:expansion embedding} below may be compared with Proposition 14 of \cite{hoehn2011}.

\begin{defn}
\label{defn:expansion}
Let $\chi$ be as above.  The \emph{$\chi$-expansion} of $G,\omega$ is the subdivision $G^+$ of $G$ together with the placement function $\omega^+ \colon \mathsf{V}(G^+) \to \Gamma$ defined by the following property:

For any edge $\{u,v\}$ in $G$, assuming $\omega(u) = o$ and $\omega(v) = b(i,t)$, if $z_0,\ldots,z_\kappa$ are the consecutive vertices in $G^+$ as in \ref{edge subdivision} of Definition~\ref{defn:subdivision}, then:
\begin{enumerate}[label=(\arabic{*})]
\item $\kappa$ equals the number of vertices on leg $i$ of $G_\chi$ (excluding the branch vertex);
\item for each $p \in \{0,\ldots,\kappa-1\}$, $\omega^+(z_p) = \chi(G_\chi(i,p))$; and
\item for the endpoint $e = G_\chi(i,\kappa)$ of leg $i$ in $G_\chi$, if $i(e)$ is such that $\chi(e) = b(i(e),1)$, then $\omega^+(v) = b(i(e),t)$.
\end{enumerate}
\end{defn}

\begin{prop}
\label{prop:expansion embedding}
Let $\chi$ be as above.  Let $T \subset \mathbb{R}^2$ be a $\Gamma$-space with $\Gamma$-marking $\mathfrak{m}_T \colon \Gamma \to T$.  Let $\varepsilon_1,\varepsilon_2 > 0$.  Suppose:
\begin{enumerate}[label=(\roman{*})]
\item \label{Omega1} $\Omega_1$ is a $\langle \chi,T,\varepsilon_1 \rangle$-embedding of $G_\chi$ with corresponding projection $\pi_1 \colon \Omega_1(G_\chi) \to T$;
\item $\Omega_1(G_\chi)$ is a $\Gamma$-space with $\Gamma$-marking $\mathfrak{m}_1 \colon \Gamma \to \Omega_1(G_\chi)$;
\item \label{b(i,t) align} for each $i \in \{0,\ldots,n\}$ and $t \in [0,1]$, $\pi_1(\mathfrak{m}_1(b(i,t))) = \mathfrak{m}_T(b(i(e),t))$, where $e$ is the endpoint of leg $i$ of $G_\chi$ and $i(e) \in \{0,\ldots,n\}$ is such that $\chi(e) = b(i(e),1)$; and
\item \label{Omega2} $\Omega_2$ is a $\langle \omega,\Omega_1(G_\chi),\varepsilon_2 \rangle$-embedding of $G$ with corresponding projection $\pi_2 \colon \Omega_2(G) \to \Omega_1(G_\chi)$.
\end{enumerate}
Let $G^+$ together with $\omega^+ \colon \mathsf{V}(G^+) \to \Gamma$ be the $\chi$-expansion of $G,\omega$.

The vertices added to $G$ to form $G^+$ can be chosen in such a way that $\Omega_2$ is a $\langle \omega^+,T,\varepsilon_1+\varepsilon_2 \rangle$-embedding of $G^+$, with corresponding projection $\pi_1 \circ \pi_2$.
\end{prop}

To help understand the idea behind Proposition~\ref{prop:expansion embedding}, consider the following example in the case $n = 3$, which is illustrated in Figure~\ref{fig:expansion}.  Let $T = T_0$, and let $G_\chi$ be the simple $4$-od graph with 3 vertices on leg 0, 5 on leg 1, 7 on leg 2, and 3 on leg 3 (including the branch vertex in each case).  The placement function $\chi \colon \mathsf{V}(G_\chi) \to \Gamma$ is defined as follows.  The images under $\chi$ of the vertices on each leg, including the branch vertex, in order, are:
\begin{align*}
\textrm{Leg 0: } & b(3,1),\ o,\ b(0,1) \\
\textrm{Leg 1: } & b(3,1),\ o,\ b(1,0),\ o,\ b(0,1) \\
\textrm{Leg 2: } & b(3,1),\ o,\ b(2,\tfrac{1}{2}),\ o,\ b(1,1),\ o,\ b(0,1) \\
\textrm{Leg 3: } & b(3,1),\ o,\ b(2,1)
\end{align*}
From the graph $G$, only one edge $\{u,v\}$ is depicted, and the placement function $\omega \colon \mathsf{V}(G) \to \Gamma$ is such that $\omega(u) = o$ and $\omega(v) = b(2,\frac{1}{3})$.  Figure~\ref{fig:expansion} shows the embedded graphs, and the vertices added to the edge $uv$ in $G$ to form $G^+$ are indicated.

\begin{figure}
\begin{center}
\includegraphics{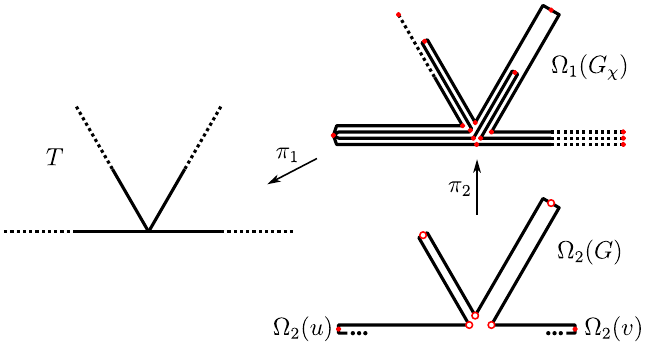}
\end{center}

\caption{An example illustrating the scenario described in Proposition~\ref{prop:expansion embedding}.  In the $\Gamma$-spaces $T$ and $\Omega_1(G_\chi)$, the sets of points marked $b(i,t)$ are represented by the dotted lines.  The images of the vertices of $G_\chi$ under $\Omega_1$ are marked with solid red dots, as are the images of the vertices $u,v \in \mathsf{V}(G)$ under $\Omega_2$.  The vertices added to the edge $uv$ in $G$ to form $G^+$ are suggested with hollow dots.  In actuality, the three pictures should be superimposed on one another, but we separate them for visual clarity.}
\label{fig:expansion}
\end{figure}

\begin{proof}
Fix any edge $\{u,v\}$ in $G$, and assume $\omega(u) = o$ and $\omega(v) = b(i,t)$.  Let $\kappa$ be the number of vertices on leg $i$ of $G_\chi$ (excluding the branch vertex).  By Definition~\ref{defn:embedding}, the restriction $\pi_2 {\restriction}_{\Omega_2(uv)}$ of $\pi_2$ to the arc $\Omega_2(uv)$ is a homeomorphism onto the arc in $\Omega_1(G_\chi)$ between $\mathfrak{m}_1(o) = \Omega_1(G_\chi(\cdot,0))$ and $\mathfrak{m}_1(b(i,t))$.  For each $p \in \{0,\ldots,\kappa-1\}$, define
\[ z_p = \left( \Omega_2^{-1} \circ (\pi_2 {\restriction}_{\Omega_2(uv)})^{-1} \circ \Omega_1 \right)(G_\chi(i,p)) ,\]
and put $\omega^+(z_p) = \chi(G_\chi(i,p))$.  Also, define $z_\kappa = v$, and put $\omega^+(v) = b(i(e),t)$, where $e = G_\chi(i,\kappa)$ is the endpoint of leg $i$ of $G_\chi$ and $i(e) \in \{0,\ldots,n\}$ is such that $\chi(e) = b(i(e),1)$.  After adding vertices $z_p$ in this manner on each edge of $G$, the resultant subdivision $G^+$, together with the function $\omega^+$, is the $\chi$-expansion of $G,\omega$.

Continuing with $u$, $v$, and $\kappa$ as above, observe that for any $p \in \{0,\ldots,\kappa-1\}$,
\[ \pi_1 \circ \pi_2(\Omega_2(z_p)) = \pi_1(\Omega_1(G_\chi(i,p))) = \mathfrak{m}_T(\chi(G_\chi(i,p))) = \mathfrak{m}_T(\omega^+(z_p)) .\]
Also,
\begin{align*}
\pi_1 \circ \pi_2(\Omega_2(z_\kappa)) &= \pi_1 \circ \pi_2(\Omega_2(v)) \\
&= \pi_1(\mathfrak{m}_1(\omega(v))) \\
&= \pi_1(\mathfrak{m}_1(b(i,t))) \\
&= \mathfrak{m}_T(b(i(e),t))  \qquad \textrm{by assumption \ref{b(i,t) align}} \\
&= \mathfrak{m}_T(\omega^+(z_p)) ,
\end{align*}
where $e$ is the endpoint of leg $i$ of $G_\chi$ and $i(e) \in \{0,\ldots,n\}$ is such that $\chi(e) = b(i(e),1)$.

The restriction of $\pi_1 \circ \pi_2$ to each arc $\Omega_2(z_p z_{p+1})$, $p \in \{0,\ldots,\kappa-1\}$, is a homeomorphism onto its image, since $\pi_2$ maps the arc $\Omega_2(z_p z_{p+1})$ one-to-one into the arc in $\Omega_1(G_\chi)$ joining $\Omega_1(G_\chi(i,p))$ and $\Omega_1(G_\chi(i,p+1))$, and the restriction of $\pi_1$ to this latter arc is a homeomorphism to the arc in $T$ joining the points $\mathfrak{m}_T(\chi(G_\chi(i,p)))$ and $\mathfrak{m}_T(\chi(G_\chi(i,p+1)))$.  Finally, for any $x \in \Omega_2(G)$,
\[ \|x - (\pi_1 \circ \pi_2)(x)\| \leq \|x - \pi_2(x)\| + \|\pi_2(x) - \pi_1(\pi_2(x))\| < \varepsilon_2 + \varepsilon_1 ,\]
by assumptions \ref{Omega1} and \ref{Omega2}.  Thus, $\Omega_2$ is a $\langle \omega^+,T,\varepsilon_1+\varepsilon_2 \rangle$-embedding of $G^+$, with corresponding projection $\pi_1 \circ \pi_2$.
\end{proof}

\begin{prop}
\label{prop:expansion cover}
Let $\chi$ be as above, and let $\delta > 0$.  Assume $G$ is connected, and let $G^+$ together with $\omega^+ \colon \mathsf{V}(G^+) \to \Gamma$ be the $\chi$-expansion of $G,\omega$.  Also, assume there exists a $\delta$-combinatorial $n$-od cover $f^+$ for $\omega^+$ such that for each $u,v \in \mathsf{V}(G^+)$ with $f^+(u) = f^+(v)$, one of the following holds:
\begin{enumerate}[label=(\Alph{*})]
\item \label{G o alone} $u,v \in \mathsf{V}(G)$ and $\omega(u) = \omega(v) = o$; or
\item \label{G b(i,.) alone} $u,v \in \mathsf{V}(G)$, and there exists $i \in \{0,\ldots,n\}$ such that $\omega(u) = b(i,\cdot)$ and $\omega(v) = b(i,\cdot)$; or
\item \label{G+ new alone} $u,v \in \mathsf{V}(G^+) \smallsetminus \mathsf{V}(G)$.
\end{enumerate}
Then there exists a $\delta$-combinatorial $n$-od cover for $\omega$.
\end{prop}

\begin{proof}
Let $C^{\mathsf{V}} = f^+(\mathsf{V}(G))$ be the set of all images under $f^+$ of the vertices from $G$ in $G^+$.  Define the function $g \colon C^{\mathsf{V}} \to \{o\} \cup \{1,\ldots,n\}$ by $g(f^+(u)) = o$ if $\omega(u) = o$, and $g(f^+(u)) = i$ if $\omega(u) = b(i,\cdot)$, for each $u \in \mathsf{V}(G)$.  By assumptions \ref{G o alone} and \ref{G b(i,.) alone}, this function $g$ is well-defined.

In order to define the $\delta$-combinatorial $n$-od cover for $\omega$, we first define the set $S \subset C^{\mathsf{V}}$, which will be equal to $f^+(f^{-1}(C(\cdot,0)))$, in two cases as follows.  If $C(\cdot,0) \in C^{\mathsf{V}}$, then $S = \{C(\cdot,0)\}$.  Otherwise, a vertex $C(\ell,j) \in C^{\mathsf{V}}$ belongs to $S$ if and only if $g(C(\ell,j)) = o$ and $C(\ell,j') \notin C^{\mathsf{V}}$ for each $0 \leq j' < j$.

Now define $f \colon \mathsf{V}(G) \to C$ as follows.  Given $u \in \mathsf{V}(G)$, if $f^+(u) \in S$, then let $f(u) = C(\cdot,0)$.  If $f^+(u) = C(\ell,j) \notin S$, then let $f(u) = C(\ell,A_u)$, where $A_u$ is the number of vertices in the set
\[ C^{\mathsf{V}} \cap \{C(\ell,j'): 0 \leq j' \leq j\} \smallsetminus S .\]

We now prove that $f$ satisfies properties \ref{C1}, \ref{C2}, and \ref{C3} of Definition~\ref{defn:comb cover}.

\medskip
\noindent \ref{C1}.  Let $u,v \in \mathsf{V}(G)$ be such that $f(u) = f(v)$.  If $f^+(u) = f^+(v)$, then by assumptions \ref{G o alone} and \ref{G b(i,.) alone}, either $\omega(u) = \omega(v) = o$, or there exists $i \in \{0,\ldots,n\}$ such that $\omega(u) = b(i,\cdot)$ and $\omega(v) = b(i,\cdot)$.  Suppose now that $f^+(u) \neq f^+(v)$.  By the definition of $f$, this can only happen if $f^+(u),f^+(v) \in S$ (so $f(u) = C(\cdot,0)$) and $S \neq \{C(\cdot,0)\}$.  In this case, by the definition of $S$, $g(f^+(u)) = g(f^+(v)) = o$, which means $\omega(u) = \omega(v) = o$.

\medskip
\noindent \ref{C2}.  Let $u,v \in \mathsf{V}(G)$ be adjacent vertices.  Assume without loss of generality that $\omega(u) = o$ and $\omega(v) = b(i,\cdot)$ for some $i \in \{0,\ldots,n\}$.  Notice that, by assumptions \ref{G o alone} and \ref{G b(i,.) alone}, $f^+(u) \neq f^+(v)$.

Let $z_0,\ldots,z_\kappa$ be the consecutive vertices in $G^+$ as in \ref{edge subdivision} of Definition~\ref{defn:subdivision}.  It follows from property \ref{C2} for $f^+$ that each vertex of $C$ strictly between $f^+(u)$ and $f^+(v)$ is the image of at least one of the vertices $z_p$, $p \in \{1,\ldots,\kappa-1\}$.  Therefore, by assumptions \ref{G o alone}, \ref{G b(i,.) alone}, and \ref{G+ new alone}, we see that no vertex of $G$ is mapped by $f^+$ strictly between $f^+(u)$ and $f^+(v)$.  It now follows from the definition of $f$ that $f(u)$ and $f(v)$ are adjacent in $C$.  We remark here that if $f^+(u)$ and $f^+(v)$ are not on the same leg of $C$, then $f^+(u) \in S$, so $f(u) = C(\cdot,0)$ and $f(v) = C(\ell,1)$ for some $\ell \in \{1,\ldots,n\}$.

\medskip
\noindent \ref{C3}.  Let $v_1,v_2,v_3 \in \mathsf{V}(G)$ be consecutive vertices in $G$, let $v \in \mathsf{V}(G) \smallsetminus \{v_1,v_2,v_3\}$, and suppose that $f(v_1) \neq f(v_3)$, $f(v) = f(v_2)$ and $f(v_2) \neq C(\cdot,0)$.  Suppose also that $0 \leq s < t \leq 1$, and that for some $i \in \{0,\ldots,n\}$, $\omega(v_2) = b(i,s)$ and $\omega(v) = b(i,t)$.

Observe that $f^+(v_2) \neq C(\cdot,0)$ since $f(v_2) \neq C(\cdot,0)$.  Also, $f^+(v) = f^+(v_2)$ since $f(v) = f(v_2)$, and by the definition of $f$.

Let $y_1,y_3 \in \mathsf{V}(G^+)$ be the vertices in $G^+$ which are adjacent to $v_2$, and such that $y_1$ is between $v_1$ and $v_2$, and $y_3$ is between $v_2$ and $v_3$ (so $y_1,v_2,y_3$ are consecutive in $G^+$).  It follows from assumptions \ref{G o alone}, \ref{G b(i,.) alone}, and \ref{G+ new alone} that $f^+(y_1)$ and $f^+(v_1)$ are in the same component of $C - f^+(v_2)$; and likewise for $f^+(y_3)$ and $f^+(v_3)$.  Given that $f$ satisfies \ref{C2}, from the assumption that $f(v_1) \neq f(v_3)$, and by the definition of $f$, it follows easily that $f^+(y_1) \neq f^+(y_3)$.

Recall from Definition~\ref{defn:expansion} that $\omega^+(v_2) = b(i(e),s)$ and $\omega^+(v) = b(i(e),t)$, where $e$ is the endpoint of leg $i$ in $G_\chi$ and $i(e)$ is such that $\chi(e) = b(i(e),1)$.  We now conclude that $t - s < \delta$ by property \ref{C3} for $f^+$.

This completes the proof that $f$ is a $\delta$-combinatorial $n$-od cover for $\omega$.
\end{proof}

\section{Generalized example of Ingram}
\label{sec:Ingram ex}

We now apply the machinery developed in the previous sections to give a new family of examples of tree-like plane continua, inspired by Ingram's example from \cite{ingram1972}.

\begin{thm}
\label{thm:Ingram ex}
For each $n \geq 3$, there exists an indecomposable arc continuum in the plane which is a simple $(n+1)$-od-like continuum but is not a simple $n$-od-like continuum.
\end{thm}

Before we begin the proof of Theorem~\ref{thm:Ingram ex}, we carry out some initial preparation.

\begin{defn}
\label{defn:Ingram graph-word}
We define \emph{Ingram's graph-word} as an $(n+1)$-od graph $G_\chi$ along with a placement function $\chi \colon \mathsf{V}(G_\chi) \to \Gamma$ such that, for each $i \in \{0,\ldots,n\}$, the leg $i$ contains $2n+2$ vertices (excluding the branch vertex), and:
\begin{enumerate}[label=(\arabic{*})]
\item $\chi(G_\chi(i,0)) = \chi(G_\chi(i,2n+2)) = b(0,1)$;
\item for each even $p = 2q \in \{2,\ldots,2n\}$, $\chi(G_\chi(i,p)) = b(q,1 - \frac{i}{n})$; and
\item for each odd $p = 2q+1 \in \{1,\ldots,2n+1\}$, $\chi(G_\chi(i,p)) = o$.
\end{enumerate}
\end{defn}
Observe that the vertices in leg $i$ (including the branch vertex) of Ingram's graph-word
\[ G_\chi(i,0),\ G_\chi(i,1),\ \ldots,\ G_\chi(i,2n+2) \]
are a wrapping pattern.  Note also that Ingram's graph-word satisfies the conditions \ref{chi1} and \ref{chi2} for $\chi$ which were assumed in the previous section.

The exact pattern of Ingram's example differs from ours in the $n=2$ case, in that the final two vertices on each leg, labelled $o$, $b(0,1)$ by $\chi$, are absent (and also the orientation is reversed).  We include these vertices in our example for convenience, so that a complete wrapping pattern is present on each leg, allowing us to apply the theory developed in Section~\ref{sec:wrapping patterns}.

With the following Lemma, we will be ready to apply Proposition~\ref{prop:expansion cover} later in this section.

\begin{lem}
\label{lem:Ingram expansion condition}
Let $G$ be a connected graph with placement function $\omega \colon \mathsf{V}(G) \to \Gamma$ and let $G^+$ together with $\omega^+ \colon \mathsf{V}(G^+) \to \Gamma$ be the $\chi$-expansion of $G,\omega$, where $G_\chi,\chi$ is Ingram's graph-word.  Let $0 < \delta < \frac{1}{n}$, and suppose $f$ is a $\delta$-combinatorial $n$-od cover for $\omega^+$.  For each $u,v \in \mathsf{V}(G^+)$, if $f(u) = f(v)$, then one of the following holds:
\begin{enumerate}[label=(\alph{*})]
\item $u,v \in \mathsf{V}(G)$ and $\omega(u) = \omega(v) = o$; or
\item $u,v \in \mathsf{V}(G)$, and there exists $i \in \{0,\ldots,n\}$ such that $\omega(u) = b(i,\cdot)$ and $\omega(v) = b(i,\cdot)$; or
\item $u,v \in \mathsf{V}(G^+) \smallsetminus \mathsf{V}(G)$.
\end{enumerate}
\end{lem}

\begin{proof}
Because each leg of Ingram's graph-word is a wrapping pattern, the entire placement function $\omega^+ \colon \mathsf{V}(G^+) \to \Gamma$ is a wrapping complex under $f$.  Specifically, let $Q$ be the graph whose vertices are the edges of $G$, and such that two vertices of $Q$ are adjacent in $Q$ if and only if the corresponding edges in $G$ share a vertex in $G$.  For each edge $q = \{u,v\}$ in $G$, assuming $\omega(u) = o$ and $\omega(v) = b(i,\cdot)$ for some $i \in \{0,\ldots,n\}$, let $k(q) = 2n+2$, and let $z_0^{(q)},\ldots,z_{k(q)}^{(q)}$ be the consecutive vertices in $G^+$ as in \ref{edge subdivision} of Definition~\ref{defn:subdivision} (with $z_0^{(q)} = u$ and $z_{k(q)}^{(q)} = v$).  Then each vertex in $G^+$ is in one of the wrapping patterns in $\left\{ \mathcal{Z}_q = \{z_0^{(q)},\ldots,z_{k(q)}^{(q)}\}: q \in \mathsf{V}(Q) \right\}$.

Let $z_{p_1}^{(q_1)}$ and $z_{p_2}^{(q_2)}$ be two arbitrary vertices in $G^+$, and assume that $f \left( z_{p_1}^{(q_1)} \right) = f \left( z_{p_2}^{(q_2)} \right)$.  According to Proposition~\ref{prop:complex sync}\ref{complex sync 1}, either $\omega^+ \left( z_{p_1}^{(q_1)} \right) = \omega^+ \left( z_{p_2}^{(q_2)} \right) = o$ and $f \left( z_{p_1}^{(q_1)} \right) = C(\cdot,0)$, or $p_1 = p_2$.  Note that $\omega^+(u) = b(0,\cdot)$ for each $u \in \mathsf{V}(G)$, so if $\omega^+ \left( z_{p_1}^{(q_1)} \right) = \omega^+ \left( z_{p_2}^{(q_2)} \right) = o$ then $z_{p_1}^{(q_1)},z_{p_2}^{(q_2)} \in \mathsf{V}(G^+) \smallsetminus \mathsf{V}(G)$.  If $p_1 = p_2 = 0$, then $z_0^{(q_1)},z_0^{(q_2)} \in \mathsf{V}(G)$ and $\omega \left( z_0^{(q_1)} \right) = \omega \left( z_0^{(q_2)} \right) = o$.  If $p_1 = p_2 \in \{1,\ldots,2n+1\}$, then $z_0^{(q_1)},z_0^{(q_2)} \in \mathsf{V}(G^+) \smallsetminus \mathsf{V}(G)$.  If $p_1 = p_2 = 2n+2$, then $z_{2n+2}^{(q_1)},z_{2n+2}^{(q_2)} \in \mathsf{V}(G)$, and $\omega \left( z_{2n+2}^{(q_1)} \right) = b(i_1,\cdot)$ and $\omega \left( z_{2n+2}^{(q_2)} \right) = b(i_2,\cdot)$ for some $i_1,i_2 \in \{0,\ldots,n\}$.  It remains to prove that $i_1 = i_2$ in this last case.  This will be accomplished in \ref{i sync 2} below.

Towards this end, we need some preliminary considerations.  Let $Q_0$ be the subgraph of $Q$ generated by the set of all vertices $q \in \mathsf{V}(Q)$ such that the wrapping pattern $\mathcal{Z}_q$ goes across the branch.  For each $q \in \mathsf{V}(Q_0)$, let $j_1(q)$ and $j_2(q)$ be as in Proposition~\ref{prop:L-covered}\ref{L across} for $\mathcal{Z}_q$.
\begin{enumerate}[label=(\arabic{section}.\arabic{thm}-\roman{*})]
\item Either $j_1(q) > 2$ for all $q \in \mathsf{V}(Q_0)$, or $j_2(q) > 2$ for all $q \in \mathsf{V}(Q_0)$.
\end{enumerate}

Indeed, suppose $j_1(q_0) = 2$ for some $q_0 \in \mathsf{V}(Q_0)$.  This implies $f \left( z_2^{(q_0)} \right) = C(\cdot,0)$, and it immediately follows that each branch-star trajectory consists of three vertices, labelled $o$, $b(1,\cdot)$, $o$.  We conclude that $j_1(q) = 2$ and $j_2(q) = 2n > 2$ for all $q \in \mathsf{V}(Q_0)$.  Suppose, now, that $j_1(q_0) = 1$ for some $q_0 \in \mathsf{V}(Q_0)$.  In this case, $i^\ast \neq 0$, and the number of $b(i^\ast,\cdot)$'s in $\mathcal{Z}_{q_0}$ is $0$, hence the number of $b(i^\ast,\cdot)$'s in $\mathcal{Z}_q$ is $0$ for all $q \in \mathsf{V}(Q_0)$ by Proposition~\ref{prop:complex sync}\ref{complex sync 2}.  This means that for each $q \in \mathsf{V}(Q_0)$, the vertex in $\mathcal{Z}_q$ labelled $b(i^\ast,\cdot)$ comes after the branch-star trajectory, which can only happen if $j_2(q) > 2$.

We now proceed by induction on the number of edges in $G$, $|\mathsf{E}(G)| = |\mathsf{V}(Q)|$.  We will prove that:

\begin{enumerate}[resume, label=(\arabic{section}.\arabic{thm}-\roman{*})]
\item \label{i sync 1} If $j_1(q) > 2$ for all $q \in \mathsf{V}(Q_0)$, then there exists $i_0 \in \{0,\ldots,n\}$ such that $\omega \left( z_{2n+2}^{(q)} \right) = b(i_0,\cdot)$ for each $q \in \mathsf{V}(Q_0)$; and
\item \label{i sync 2} For all $q_1,q_2 \in \mathsf{V}(Q)$, if $f \left( z_{2n+2}^{(q_1)} \right) = f \left( z_{2n+2}^{(q_2)} \right)$, then there exists $i \in \{0,\ldots,n\}$ such that $\omega \left( z_{2n+2}^{(q_1)} \right) = b(i,\cdot)$ and $\omega \left( z_{2n+2}^{(q_2)} \right) = b(i,\cdot)$.
\end{enumerate}

If $|\mathsf{E}(G)| = 1$, then \ref{i sync 1} and \ref{i sync 2} are vacuous.

Now suppose $|\mathsf{E}(G)| > 1$, and that \ref{i sync 1} and \ref{i sync 2} hold for graphs with fewer than $|\mathsf{E}(G)|$ vertices.

We first prove \ref{i sync 1} for $G$.  Suppose $j_1(q') > 2$ for all $q' \in \mathsf{V}(Q_0)$.  Let $q = \{u,v\}$ be an edge in $G$ for which $G - q$ is connected, where $\omega(u) = o$, and $\omega(v) = b(i,\cdot)$ for some $i \in \{0,\ldots,n\}$ (so $u = z_0^{(q)}$ and $v = z_{2n+2}^{(q)}$).  Recall that, according to our notion of edge removal, $u \in \mathsf{V}(G - q)$ (respectively, $v \in \mathsf{V}(G - q)$) if and only if it is adjacent to another vertex other than $v$ (respectively, other than $u$) in $G$.  At least one of $u$ and $v$ belongs to $V(G - q)$.  We may assume that the wrapping pattern $\mathcal{Z}_q$ goes across the branch (otherwise there is nothing to prove, by induction), and that there is at least one other wrapping pattern in the wrapping complex which goes across the branch, i.e.\ that $\mathsf{V}(Q_0) \smallsetminus \{q\} \neq \emptyset$.  By induction, there exists $i_0 \in \{0,\ldots,n\}$ such that $\omega \left( z_{2n+2}^{(q_0)} \right) = b(i_0,\cdot)$ for each $q_0 \in \mathsf{V}(Q_0) \smallsetminus \{q\}$.  We must prove that $i = i_0$.  We consider two cases.

\medskip
\textbf{Case 1 (for \ref{i sync 1}):} $u \in \mathsf{V}(G - q)$.

Because $G - q$ is connected and there is at least one wrapping pattern in $G^+$ other than $\mathcal{Z}_q$ that goes across the branch, there must exist an edge $q' = \{u',v'\}$ in $G$ such that one of $f(u')$ or $f(v')$ is on the same leg as $f(u)$ in $C$, and $\mathcal{Z}_{q'}$ goes across the branch, i.e.\ $q' \in \mathsf{V}(Q_0) \smallsetminus \{q\}$.  From Proposition~\ref{prop:complex sync}\ref{complex sync 1} we conclude that $f(u') = f(u)$, i.e.\ that
\[ f \left( z_0^{(q)} \right) = f \left( z_0^{(q')} \right) = C(\ell,j_1(q)) ,\]
for some $\ell \in \{1,\ldots,n\}$.  By inductive hypothesis \ref{i sync 1}, $\omega(v') = b(i_0,\cdot)$.  Now $j_1(q) = j_1(q') > 2$ by assumption, which implies that
\begin{align*}
f \left( z_1^{(q)} \right) = f \left( z_1^{(q')} \right) &= C(\ell,j_1(q) - 1) \\
f \left( z_2^{(q)} \right) = f \left( z_2^{(q')} \right) &= C(\ell,j_1(q) - 2) \\
f \left( z_3^{(q)} \right) = f \left( z_3^{(q')} \right) &= C(\ell,j_1(q) - 3) .
\end{align*}
From the definition of Ingram's graph-word $\chi$, we see that $\omega^+ \left( z_2^{(q')} \right) = b(1,1-\frac{i_0}{n})$, and $\omega^+ \left( z_2^{(q)} \right) = b(1,1-\frac{i}{n})$.  Therefore, by property \ref{C3} of $f$, we conclude that $|\frac{i}{n} - \frac{i_0}{n}| < \delta < \frac{1}{n}$, hence $i = i_0$, as desired.

\medskip
\textbf{Case 2 (for \ref{i sync 1}):} $v \in \mathsf{V}(G - q)$.

By the same reasoning as in Case 1, there exists an edge $q' = \{u',v'\} \in \mathsf{V}(Q_0) \smallsetminus \{q\}$ such that $f(v') = f(v)$.  By inductive hypothesis \ref{i sync 1}, $\omega(v') = b(i_0,\cdot)$.  Since $v,v' \in \mathsf{V}(G - q)$, by inductive hypothesis \ref{i sync 2}, we have that $\omega(v) = b(i_0,\cdot)$ as well, i.e.\ $i = i_0$.

\medskip
We now prove \ref{i sync 2}.  Let $q_1 = \{u_1,v_1\}$ be an edge in $G$ for which $G - q_1$ is connected, where $\omega(u_1) = o$, and $\omega(v_1) = b(i_1,\cdot)$ for some $i_1 \in \{0,\ldots,n\}$ (so $u_1 = z_0^{(q_1)}$ and $v_1 = z_{2n+2}^{(q_1)}$).  Let $q_2 = \{u_2,v_2\} \in \mathsf{V}(Q) \smallsetminus \{q_1\}$, where $\omega(u_2) = o$, and $\omega(v_2) = b(i_2,\cdot)$ for some $i_2 \in \{0,\ldots,n\}$ (so $u_2 = z_0^{(q_2)}$ and $v_2 = z_{2n+2}^{(q_2)}$), and suppose that $f \left( z_{2n+2}^{(q_1)} \right) = f \left( z_{2n+2}^{(q_2)} \right)$, i.e.\ that $f(v_1) = f(v_2)$.  By induction, it suffices show that $i_1 = i_2$.  Clearly, we may assume $v_1 \neq v_2$.  We again consider the same two cases as above.

\medskip
\textbf{Case 1 (for \ref{i sync 2}):} $u_1 \in \mathsf{V}(G - q_1)$.

Because $G - q_1$ is connected and both $f(u_1)$ and $f(v_1) = f(v_2)$ belong to the image of $\mathsf{V}(G - q_1)$ under $f$, we see from Proposition~\ref{prop:complex sync}\ref{complex sync 1} that either there exists an edge $q' = \{u',v'\} \in \mathsf{V}(Q) \smallsetminus \{q_1\}$, where $u',v' \in \mathsf{V}(G - q_1)$, such that $f(u') = f(u_1)$ and $f(v') = f(v_1)$, or, in the case that $\mathcal{Z}_{q_1}$ goes across the branch, there exist two edges (possibly the same) $q' = \{u',v'\}, q'' = \{u'',v''\} \in \mathsf{V}(Q_0) \smallsetminus \{q_1\}$, where $u',u'',v',v'' \in \mathsf{V}(G - q_1)$, such that $f(u'') = f(u_1)$ and $f(v') = f(v_1)$.  In either case, by inductive hypothesis \ref{i sync 2} we have that $\omega(v') = b(i_2,\cdot)$.  In the latter case, if additionally $j_1(q_0) > 2$ for each $q_0 \in \mathsf{V}(Q_0)$, then $i_1 = i_2 = i_0$ by \ref{i sync 1}.  In the remaining cases, we have that
\begin{align*}
f \left( z_{2n+1}^{(q_1)} \right) &= f \left( z_{2n+1}^{(q')} \right) \\
f \left( z_{2n}^{(q_1)} \right) &= f \left( z_{2n}^{(q')} \right) \\
f \left( z_{2n-1}^{(q_1)} \right) &= f \left( z_{2n-1}^{(q')} \right) ,
\end{align*}
and these three are consecutive vertices, in the same leg of $C$ as $f(v_1)$.  From the definition of Ingram's graph-word $\chi$, we see that $\omega^+ \left( z_{2n}^{(q')} \right) = b(n,1-\frac{i_2}{n})$ and $\omega^+ \left( z_{2n}^{(q)} \right) = b(n,1-\frac{i_1}{n})$.  Therefore, by property \ref{C3} of $f$, we conclude that $|\frac{i_2}{n} - \frac{i_1}{n}| < \delta < \frac{1}{n}$, hence $i_1 = i_2$.

\medskip
\textbf{Case 2 (for \ref{i sync 2}):} $v \in \mathsf{V}(G - q_1)$.

In this case, it follows immediately from inductive hypothesis \ref{i sync 2} that $i_1 = i_2$, since $v_1,v_2 \in \mathsf{V}(G - q_1)$.
\end{proof}

\begin{lem}
\label{lem:no comb cover}
Let $0 < \delta < \frac{1}{n}$.  There does not exist a $\delta$-combinatorial $n$-od cover for $\chi$.
\end{lem}

\begin{proof}
Let $S$ be the simple $(n+1)$-od graph whose vertex set consists of one branch vertex $S(\cdot,0)$ and $n+1$ endpoints $S(0,1),\ldots,S(n,1)$, each connected to $S(\cdot,0)$ by an edge.  Define the placement function $\iota \colon \mathsf{V}(S) \to \Gamma$ by $\iota(S(\cdot,0)) = o$, and $\iota(S(i,1)) = b(i,1)$ for each $i \in \{0,\ldots,n\}$.  Notice that there is no $\delta$-combinatorial $n$-od cover for $\iota$ by \ref{C1}, \ref{C2}, and the pigeonhole principle.

Observe that $G_\chi,\chi$ itself can be viewed as the $\chi$-expansion of $S,\iota$.  Therefore, by Proposition~\ref{prop:expansion cover} and Lemma~\ref{lem:Ingram expansion condition}, since $\iota$ has no $\delta$-combinatorial $n$-od cover, $\chi$ cannot have one either.
\end{proof}

We are now ready to proceed with the proof of Theorem~\ref{thm:Ingram ex}.

\begin{proof}[Proof of Theorem~\ref{thm:Ingram ex}]
Let $0 < \delta < \frac{1}{n}$ be sufficiently small as required in Proposition~\ref{prop:cover to comb cover}, and let $\varepsilon > 0$ be such that $2\varepsilon < \delta$.  To produce the continuum $X$, we recursively construct the following, for $N = 0,1,\ldots$:
\begin{itemize}
\item embeddings $\Omega_N \colon G_\chi \to \mathbb{R}^2$;
\item maps $\pi_N \colon T_{N+1} \to T_N$, where $T_N = \Omega_N(G_\chi)$;
\item real numbers $\varepsilon_N > 0$; and
\item simple $(n+1)$-od covers $\mathcal{U}_N$ of $T_N$ of mesh less than $\varepsilon_N$,
\end{itemize}
satisfying the following properties:
\begin{enumerate}[label=(\roman{*})]
\item $T_N$ is a $\Gamma$-space, with $\Gamma$-marking $\mathfrak{m}_N \colon \Gamma \to T_N$;
\item for each $i \in \{0,\ldots,n\}$ and $t \in [0,1]$, $\pi_N(\mathfrak{m}_{N+1}(b(i,t))) = \mathfrak{m}_N(b(0,t))$;
\item $\Omega_{N+1}$ is a $\langle \chi,T_N,\varepsilon_N \rangle$-embedding with corresponding projection $\pi_N$;
\item the elements of $\mathcal{U}_N$ are open subsets of $\mathbb{R}^2$, and $\overline{\bigcup \mathcal{U}_{N+1}} \subseteq \bigcup \mathcal{U}_N$; and
\item $\sum_{N=0}^\infty \varepsilon_N \leq \varepsilon$.
\end{enumerate}

To begin, let $T_0$ be the $\Gamma$-space defined earlier in this paper, and let $\mathfrak{m}_0 = \mathfrak{m}_{T_0}$.  Let $\Omega_0 \colon G_\chi \to T_0$ be a homeomorphism, let $\varepsilon_0 = \frac{\varepsilon}{2}$, and let $\mathcal{U}_0$ be a simple $(n+1)$-od cover of $T_0$ of mesh less than $1$.

Now, given $N \geq 0$, suppose that $\Omega_N$, $\varepsilon_N$, and $\mathcal{U}_N$ have been constructed.  It is easy to see from the definition of Ingram's graph-word $\chi$ that for any $\varepsilon > 0$, there exists a $\langle \chi,T_0,\varepsilon \rangle$-embedding $\Omega'$ of $G_\chi$ with projection $\pi' \colon \Omega'(G_\chi) \to T_0$ such that $\Omega'(G_\chi)$ is a $\Gamma$-space with $\Gamma$-marking $\mathfrak{m}'$, and for each $i \in \{0,\ldots,n\}$ and all $t \in [0,1]$, $\pi' \circ \mathfrak{m}'(b(i,t)) = b(0,t)$.  See Figure~\ref{fig:Ingram ex} for an illustration in the case $n = 3$.  Therefore, by Proposition~\ref{prop:embed G}, there exists a $\langle \chi,T_N,\varepsilon_N \rangle$-embedding $\Omega_{N+1}$ of $G_\chi$ with projection $\pi_N \colon T_{N+1} = \Omega_{N+1}(G_\chi) \to T_N$ such that $T_{N+1}$ is a $\Gamma$-space with $\Gamma$-marking $\mathfrak{m}_{N+1}$, and for each $i \in \{0,\ldots,n\}$ and all $t \in [0,1]$, $\pi_N \circ \mathfrak{m}_{N+1}(b(i,t)) = \mathfrak{m}_N(b(0,t))$, and such that $T_{N+1}$ is contained in $\bigcup \mathcal{U}_N$.

\begin{figure}
\begin{center}
\includegraphics{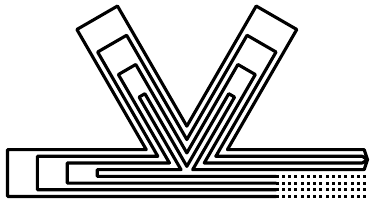}
\end{center}

\caption{Illustration of the second approximation to the generalized Ingram example (in the case $n = 3$); that is, an embedding of $G_\chi$ close to $T_0$.  The set of points of the form $\mathfrak{m}_1(b(i,t))$ is represented by the dotted lines.}
\label{fig:Ingram ex}
\end{figure}

Next, choose $\varepsilon_{N+1} > 0$ with $\varepsilon_{N+1} < 2^{-(N+2)} \varepsilon$ and let $\mathcal{U}_{N+1}$ be a simple $(n+1)$-od cover of $T_{N+1}$ such that $\overline{\bigcup \mathcal{U}_{N+1}} \subset \bigcup \mathcal{U}_N$.  This completes the recursive construction.

Let $X$ be the Hausdorff limit of the sequence $T_N$, $N = 0,1,\ldots$, in $\mathbb{R}^2$.  Equivalently, $X$ is the intersection $X = \bigcap_{N=0}^\infty \overline{\bigcup \mathcal{U}_N}$.  Moreover, we may carry out the above construction in such a way that the numbers $\varepsilon_N$ converge to $0$ fast enough so that $X$ is homeomorphic to the inverse limit $\varprojlim \langle T_N,\pi_N \rangle$ (cf.\ the Anderson-Choquet Embedding Theorem \cite{anderson-choquet1959}).  We immediately see that $X$ is a simple $(n+1)$-od-like continuum.

Observe that, by definition of Ingram's graph-word $\chi$, for each leg of $T_{N+2}$, its image under $\pi_{N+1}$ contains leg $0$ of $T_{N+1}$, and the image of leg $0$ of $T_{N+1}$ under $\pi_N$ is all of $T_N$.  Also, for each edge $\{u,v\}$ in $G_\chi$, $\pi_N(\Omega_{N+1}(uv))$ is equal to a leg in $T_N$.  Putting these observations together, it follows that if $Y \subset X$ is a proper subcontinuum of $X$, then, representing $X$ as the inverse limit $X = \varprojlim \langle T_N,\pi_N \rangle$ and letting $\sigma_N \colon X \to T_N$ denote the $N$-th coordinate projection of $X$ to $T_N$ for $N = 0,1,\ldots$, we have that for all but finitely many $N$, $\sigma_N(Y)$ is contained in $\Omega_N(uv)$ for some edge $\{u,v\}$ in $G_\chi$.  Since $\pi_N$ is one-to-one on each arc of the form $\Omega_{N+1}(uv)$, where $\{u,v\}$ is an edge of $G_\chi$, it follows that $Y = \varprojlim \langle \sigma_N(Y),\pi_N \rangle$ is an arc.  Moreover, $X$ cannot be the union of two proper subcontinua $Y_1,Y_2 \subset X$, since for large enough $N$, $\sigma_N(Y_1) \cup \sigma_N(Y_2)$ is a proper subset of $T_N$.  Thus, $X$ is indecomposable.

Finally, we prove that $X$ is not a simple $n$-od-like continuum.  Suppose, for a contradiction, that there is a simple $n$-od cover $\mathcal{V}$ of $X$ of mesh less than $\delta - 2\varepsilon$, where each element of $\mathcal{V}$ is an open subset of $\mathbb{R}^2$.  Then for some $N \geq 0$, $T_N \subset \bigcup \mathcal{V}$.  Define $G_N^{(N)} = G_\chi$, $\omega_N^{(N)} = \chi$.  Recursively define $G_N^{(M)}$, $\omega_N^{(M)}$, for $M \in \{1,\ldots,N-1\}$, to be the $\chi$-expansion of $G_N^{(M+1)},\omega_N^{(M+1)}$ (using Proposition~\ref{prop:expansion embedding} to choose the placement of vertices added to $G_N^{(M+1)}$ to get $G_N^{(M)}$).  In the end, by Proposition~\ref{prop:expansion embedding} applied repeatedly, $\Omega_N$ is a $\langle \omega_N^{(1)},T_0,\sum_{M=1}^N \varepsilon_M \rangle$-embedding of $G_N^{(1)}$.  Since $\sum_{N=1}^\infty \varepsilon_N \leq \varepsilon$, we see that $\Omega_N$ is a $\langle \omega_N^{(1)},T_0,\varepsilon \rangle$-embedding of $G_N^{(1)}$.

Now, by Proposition~\ref{prop:cover to comb cover}, there exists a $\delta$-combinatorial $n$-od cover of $\omega_N^{(1)}$.  By applying Proposition~\ref{prop:expansion cover} and Lemma~\ref{lem:Ingram expansion condition} repeatedly, we obtain a $\delta$-combinatorial $n$-od cover of $\omega_N^{(M)}$ for each $M \in \{1,\ldots,N\}$; in particular, for $\omega_N^{(N)} = \chi$.  However, this contradicts Lemma~\ref{lem:no comb cover}.  Therefore, $X$ is not a simple $n$-od-like continuum.
\end{proof}

\section{Discussion and questions}
\label{sec:questions}

One might expect that the methods in this paper could be adapted to answer the special case of \cite[Question 8]{lewis2007} pertaining to trees.  We leave this as an open problem:

\begin{question}
\label{ques:T2-like not T1-like}
Let $T_1$ and $T_2$ be trees such that $T_1$ is a monotone image of $T_2$, but not vice versa.  Does there exist an indecomposable $T_2$-like arc continuum in the plane which is not $T_1$-like?
\end{question}

Towards this end, it would be interesting to develop a more general notion of a combinatorial $T$-cover, for $T$ a tree.

The example of Theorem~\ref{thm:Ingram ex} above is the inverse limit of a simplicial inverse system.  This is essentially because, for a fixed $n \geq 2$, only finitely many of the symbols from $\Gamma$ are used in the placement function $\chi$ of that example.  Consequently, according to \cite{minc1994}, this example must have positive span.  Recall that a continuum $X$ has \emph{span zero} if for any continuum $C$ and any maps $f,g \colon C \to X$ with $f(C) = g(C)$, there exists a point $t \in C$ such that $f(t) = g(t)$; we say $X$ has \emph{positive span} if it does not have span zero.  Though our example from Theorem~\ref{thm:Ingram ex} has positive span, constructions within the framework developed in this paper may produce continua with span zero.  In a forthcoming paper, we will exhibit a family of continua with the same properties as in Theorem~\ref{thm:Ingram ex} above, but which also have span zero.

One may also ask a variant of Question~\ref{ques:T2-like not T1-like} above for span zero continua:
\begin{question}
Let $T_1$ and $T_2$ be as in Question~\ref{ques:T2-like not T1-like}.  Does there exist an $T_2$-like (plane) continuum with span zero which is not $T_1$-like?
\end{question}

\begin{question}
Does there exist a tree-like (plane) continuum with span zero which is not $T$-like for any tree $T$?
\end{question}

\bibliographystyle{amsplain}
\bibliography{Covers}

\end{document}